\documentclass[11pt]{amsart}
\usepackage{graphicx} 
\usepackage[utf8]{inputenc}
\usepackage{comment}
\usepackage[english]{babel}
\usepackage{enumitem}
\usepackage{amsmath,amsfonts,amssymb,amsthm}
\usepackage{amscd}
\usepackage{tikz}

\usepackage[a4paper,nomarginpar]{geometry}

\usepackage{epsfig,graphicx,color}
\usepackage[colorlinks=true, linkcolor=blue, urlcolor=red, citecolor=blue, hyperindex]{hyperref}
\usepackage[all]{xy}
\newtheorem{theorem}{Theorem}[section]
\newtheorem{lemma}[theorem]{Lemma}
\newtheorem{proposition}[theorem]{Proposition}
\newtheorem{corollary}[theorem]{Corollary}

\newtheorem{definition}[theorem]{Definition}
\newtheorem{example}[theorem]{Example}

\newtheorem*{observation*}{Observation}

\newtheorem{remark}[theorem]{Remark}

\newcommand{\N}{\mathbb{N}}  
\newcommand{\Z}{\mathbb{Z}}  
\newcommand{\R}{\mathbb{R}}  

\newcommand{\eps}{\varepsilon} 

\newcommand{\oa}{\overrightarrow}

\newcommand{\ent}{\mathrm{ent}}
\newcommand{\cantor}{\mathbf C}

\newcommand{\orb}{{\mathrm{Orb}}}
\newcommand{\fence}{{\bf F}_{\Phi}{}}
\newcommand{\nfence}{{\bf F}_{\Phi_n}{}}

\newcommand{\homeocantor}{{\mathbf H}_{X}}

\newcommand{\ffence}[2]{{\mathbf L}_{{#1}_{#2}}}
\newcommand{\groupcantor}{{\mathbb H}(\cantor)}

\newcommand{\groupfence}{{\mathbb{H}(\fence)}}
\newcommand{\cC}{\mathcal{C}}

\newcommand{\cF}{\mathcal{F}}
\newcommand{\cstructure}{\{(V_n,\Psi_n)\}_{n=0}^{\infty}}
\newcommand{\fstructure}{\{(V_n, \Psi_n,\varphi^L_n,\varphi^U_n)\}_{n=0}^{\infty}}
\newcommand{\csystem}{\{(G_n,\Psi_n)\}_{n=0}^{\infty}}
\newcommand{\fsystem}{\{(G_n, \Psi_n,\varphi^L_n,\varphi^U_n)\}_{n=0}^{\infty}}

\title{Dynamics on Fences}

\author{Jernej \v Cin\v c}

\address[J.\ \v{C}in\v{c}\footnote{https://orcid.org/0000-0001-8516-6023}]{Department of Mathematics and Computer Science, Faculty of Natural Sciences and Mathematics, University of Maribor, Koro\v ska 160, 2000 Maribor, Slovenia -- $\&$ -- Abdus Salam International Centre for Theoretical Physics (ICTP), Trieste, Italy}
\email{jernej.cinc@um.si}

\author{Udayan B. Darji} 

\address[U.\ Darji\footnote{https://orcid.org/0000-0002-2899-919X}]{Department of Mathematics, University of Louisville, Louisville, Kentucky 40292, USA}
\email{ ubdarj01@gmail.com}

\author{Benjamin Vejnar}

\address[B. Vejnar\footnote{https://orcid.org/0000-0002-2833-5385}]{Department of Mathematical Analysis,
Faculty of Mathematics and Physics, Charles University
Prague, Czechia}
\email{vejnar@karlin.mff.cuni.cz}

\date{\today}

\begin{document}

\begin{abstract}
Homeomorphisms of the Cantor set play a central role in topology, dynamical systems and descriptive set theory. In parallel, several classes of fence-like spaces—such as the hairy Cantor set, hairy arcs, Cantor bouquets in  complex dynamics, the Lelek fan in topology and Fra\"iss\'e fence in descriptive set theory—have recently been studied for their rich structural and dynamical properties.
In this paper, we introduce a general construction that associates to each homeomorphism of the Cantor set a canonically defined homeomorphism of a corresponding fence space. This construction lifts dynamical properties from the Cantor set to these fence-like spaces, allowing one to systematically transfer features such as minimality, recurrence, and orbit structure. As a consequence, we obtain a unified framework for studying dynamics on a broad class of fence-like spaces and establish new connections between their topological structure and induced dynamical behavior.
\end{abstract}

\maketitle

\tableofcontents
\section{Introduction}

The interaction between continuum theory and dynamics has produced many examples in which complex topology and dynamics coexist. This is particularly evident in spaces built from arcs with dense sets of endpoints. Such spaces arise in complex dynamics as models for Julia sets and escaping sets of transcendental entire functions, in continuum theory as canonical endpoint-dense continua, and in descriptive set theory via projective Fra\"iss\'e limits. In this paper, we introduce a flexible class of such spaces, called \emph{Scissorhand fences}, and develop a general framework for lifting dynamical systems from the Cantor space to these spaces.

\subsection{Spaces}

Further examples of this phenomenon arise in complex dynamics, where iteration of transcendental entire functions produces spaces with dense endpoint structure. A classical example was studied by Devaney and Krych \cite{devaney1984dynamics}, who considered the exponential family
\[
f_{\lambda}(z) = \lambda e^z, \quad \lambda \in \mathbb{C} \setminus \{0\}.
\]
They showed that for $0<\lambda<1/e$ the Julia set $J(f_{\lambda})$ is a union of uncountably many pairwise disjoint curves (often called \emph{hairs} or \emph{strings}), each homeomorphic to $[0,\infty)$, forming what they termed a \emph{Cantor set of curves}. Similar structures occur for other transcendental entire functions, including $z\mapsto \mu\sin z$ and $z\mapsto \mu\cos z$ for suitable parameters \cite{DevaneyTangerman}. Mayer \cite{Mayer1990}, building on work of Devaney and Goldberg \cite{DevaneyGoldberg}, showed that in these cases the set of endpoints is totally disconnected, while its union with the point at infinity is connected.

A systematic topological study of these spaces was carried out by Aarts and Oversteegen \cite{AO93}. Building on earlier work of Devaney and coauthors~\cite{Devaney1984,DevaneyTangerman,Devaney1993,devaney1984dynamics}, they considered Julia sets of the family
\[
f_{\lambda_1,\lambda_2}(z)=\lambda_1 e^z+\lambda_2 e^{-z}
\]
and showed that all Cantor sets of curves are mutually homeomorphic, both within a given family and across different families, and are in fact ambiently homeomorphic in the plane. Thus, a broad class of Julia sets arising in transcendental dynamics admits a common topological model. This rigidity is also interesting in view of their diverse measure-theoretic properties: for instance, McMullen~\cite{McMullen1987} showed that Cantor sets of curves in the exponential family have Hausdorff dimension $2$ and zero planar Lebesgue measure, whereas examples in the sine family may have positive planar Lebesgue measure~\cite{EremenkoLyubich1992,McMullen1987}.

Aarts and Oversteegen \cite{AO93} also introduced a broader class of spaces, called \emph{hairy objects}, which combine features of the Cantor set and the interval. These spaces are topologically unique and ambiently homeomorphic in the plane. Using this framework, they obtained a complete topological description of Julia sets for many maps in the exponential family.

In what follows, let $\mathbb{I}=\mathbb{R}\setminus\mathbb{Q}$. Aarts and Oversteegen \cite{AO93} introduced several canonical models for endpoint--dense structures arising in transcendental dynamics. One of them is the \emph{straight brush}, a subset $\mathcal{B}\subset [0,\infty)\times \mathbb{I}$ consisting of vertical half--lines (``hairs'') whose endpoints form a dense subset of the base. Each hair has the form
\[
h_a=[t_a,\infty)\times\{a\}, \qquad a\in\mathbb{I},
\]
and the imposed density and closure conditions on the endpoints ensure that $\mathcal{B}$ is a closed subset of the plane with a highly structured endpoint set.

A natural compactification of $\mathcal{B}$ is obtained by adding a single point at infinity; the resulting space is called a \emph{smooth Cantor bouquet}. Such spaces appeared earlier in continuum theory. In particular, Lelek \cite{Lelek1961} constructed a smooth fan with a dense set of endpoints, and later Charatonik \cite{Cha89} and Bula and Oversteegen \cite{oversteegenOneLF} independently proved its topological uniqueness. This continuum is now known in Continuum Theory as the \emph{Lelek fan} (see Figure~\ref{fig:Lelek}).

Let $\cantor$ denote the Cantor space. The \emph{Cantor fence} is any space homeomorphic to $\cantor \times I$. A general theory of fences was developed by Basso and Camerlo \cite{BassoCamerlo}, where a \emph{fence} is defined as a compact metrizable space whose connected components are either points or arcs. In this terminology, a \emph{smooth fence} is, up to homeomorphism, a compact metric subspace of the Cantor fence with components of this form (see \cite[Theorem~4.2]{BassoCamerlo}). 

An analogous representation holds for the fans discussed above: by \cite[Proposition~4]{Cha89}, every \emph{smooth fan} is, up to homeomorphism, a subcontinuum of the Cantor fan, obtained from the Cantor fence by collapsing its base to a point. A \emph{Lelek fence} is a smooth fence whose base is homeomorphic to the Cantor set and whose set of endpoints outside the base is dense; by \cite{Cha89,BO90}, the Lelek fence is unique up to homeomorphism.

	\begin{figure}[h]
    \hspace{1cm}
		\begin{minipage}{.3\textwidth}
			\begin{tikzpicture}[scale=2.5]
			\draw[thick](0,0)--(0,1);
			\draw[thick](1,0)--(1,1);
			\draw[thick](1/3,0)--(1/3,1);
			\draw[thick](2/3,0)--(2/3,1);
			\draw[thick](1/9,0)--(1/9,1);
			\draw[thick](2/9,0)--(2/9,1);
			\draw[thick](7/9,0)--(7/9,1);
			\draw[thick](8/9,0)--(8/9,1);
			
			\draw[thick](1/27,0)--(1/27,1);
			\draw[thick](2/27,0)--(2/27,1);
			\draw[thick](7/27,0)--(7/27,1);
			\draw[thick](8/27,0)--(8/27,1);
			
			\draw[thick](26/27,0)--(26/27,1);
			\draw[thick](25/27,0)--(25/27,1);
			\draw[thick](20/27,0)--(20/27,1);
			\draw[thick](19/27,0)--(19/27,1);
			\node at (1/2,-0.2) {\small  Cantor fence};
			\end{tikzpicture}
		\end{minipage}
		\begin{minipage}{.15\textwidth}
			\begin{tikzpicture}[scale=3]
			\draw[thick](0,0)--(0,1/3);
			\draw[thick](1,0)--(1,1/9);
			\draw[thick](1/3,0)--(1/3,3/4);
			\draw[thick](2/3,0)--(2/3,2/3);
			\draw[thick](1/9,0)--(1/9,3/4);
			\draw[thick](2/9,0)--(2/9,1/9);
			\draw[thick](7/9,0)--(7/9,2/7);
			\draw[thick](8/9,0)--(8/9,4/7);
			
			\draw[thick](1/27,0)--(1/27,1/2);
			\draw[thick](2/27,0)--(2/27,2/9);
			\draw[thick](7/27,0)--(7/27,8/17);
			\draw[thick](8/27,0)--(8/27,2/5);
			
			\draw[thick](26/27,0)--(26/27,3/5);
			\draw[thick](25/27,0)--(25/27,6/7);
			\draw[thick](20/27,0)--(20/27,1/4);
			\draw[thick](19/27,0)--(19/27,1/8);
			\node at (1/2,-0.2) {\small Lelek fence};
			\end{tikzpicture}
		\end{minipage}
		\hspace{2.5cm}
		\begin{minipage}{.15\textwidth}
			\begin{tikzpicture}[scale=3]
			\draw[thick](1/2,0)--(0,1/3);
			\draw[thick](1/2,0)--(1,1/9);
			\draw[thick](1/2,0)--(1/3,3/4);
			\draw[thick](1/2,0)--(2/3,2/3);
			\draw[thick](1/2,0)--(1/9,3/4);
			\draw[thick](1/2,0)--(2/9,1/9);
			\draw[thick](1/2,0)--(7/9,2/7);
			\draw[thick](1/2,0)--(8/9,4/7);
			
			\draw[thick](1/2,0)--(1/27,1/2);
			\draw[thick](1/2,0)--(2/27,2/9);
			\draw[thick](1/2,0)--(7/27,8/17);
			\draw[thick](1/2,0)--(8/27,2/5);
			
			\draw[thick](1/2,0)--(26/27,3/5);
			\draw[thick](1/2,0)--(25/27,6/7);
			\draw[thick](1/2,0)--(20/27,1/4);
			\draw[thick](1/2,0)--(19/27,1/8);
			\node at (1/2,-0.2) {\small  Lelek fan};
			\end{tikzpicture}
		\end{minipage}	
        \label{fig:Lelek}
	\end{figure}

Another compactification of $\mathcal{B}$ considered in \cite{AO93} leads to the notion of a \emph{hairy arc}. In this model, vertical hairs are attached to a base interval according to a length function $\ell:I\to I$, producing a compact subset of the unit square whose fibers are intervals of varying lengths. Since the set of endpoints is dense, this space is homeomorphic to the Lelek fence. Aarts and Oversteegen further proved that all one--sided hairy arcs (i.e.\ planar embeddings of hairy arcs) are ambiently homeomorphic, providing canonical topological models for many Julia sets of transcendental entire functions.

Recently, Cheraghi \cite{Cheraghi2025} proved that irrationally indifferent attractors satisfy a topological trichotomy: the post--critical set is either a Jordan curve, a \emph{one--sided hairy Jordan curve}, or a \emph{Cantor bouquet}. The latter two belong to the class of endpoint--dense spaces studied by Aarts and Oversteegen \cite{AO93}, providing another dynamical setting in which these canonical models arise.  

Furthermore, Cheraghi and Pedramfar \cite{haircantorset} introduced \emph{hairy Cantor sets}, which share many of the structural properties of the spaces described above. They gave an axiomatic characterization of these sets and proved that any two such planar continua are ambiently homeomorphic. By \cite[Corollary~5.6]{haircantorset}, the set of endpoints is dense, and hence the space is homeomorphic to the Lelek fence.

Spaces of the type described above also appear in topological dynamics as almost one-to-one extensions of minimal systems, beginning with the classical Floyd--Auslander systems (see, e.g., \cite{HaddadJohnson97} and references therein) and their later generalizations \cite{DeeleyPutnamStrung}, where the primary emphasis is on dynamical properties. In contrast, Balibrea, Downarowicz, Hric, Snoha, and \v Spitalsk\'y \cite{BalibreaDownarowiczHricSnohaSpitalsky} introduced \emph{cantoroids} to capture the interplay between topology and minimal dynamics. Cantoroids are almost totally disconnected spaces (i.e.\ spaces with dense degenerate components) without isolated points, and may therefore contain uncountably many non-degenerate components. The work \cite{BalibreaDownarowiczHricSnohaSpitalsky} focuses in particular on the construction of minimal (non-invertible) maps on such spaces.

Spaces as described above also arise naturally in descriptive set theory via projective Fra\"iss\'e methods. In the setting of compact metrizable spaces, Basso and Camerlo \cite{BassoCamerlo} introduced fences and smooth fences, characterized smooth fences as those obtained as limits of projective sequences of finite structures, and identified a canonical quotient of the corresponding projective Fra\"iss\'e limit, called the \emph{Fra\"iss\'e fence}. Related ideas were developed by Barto\v{s}ov\'a and Kwiatkowska \cite{BartosovaKwiatkowska2015}, who showed that the Lelek fan arises as a natural quotient of a projective Fra\"iss\'e limit of finite rooted trees and used this representation to study both the space and its homeomorphism group.  

In contrast, our approach is based on a concrete class of fences defined via semicontinuous functions; the spaces we consider form a subclass of the smooth fences studied in \cite{BassoCamerlo}. Our aim is to provide a unified framework covering a broad class of such spaces and to facilitate their dynamical analysis. To this end, we introduce \emph{Scissorhand fences}, namely smooth fences with a dense set of endpoints.

\begin{figure}
	\centering
	\begin{tikzpicture}[scale=5]
	
	\draw[thick] (0,0.289) -- (0,0.5);
	\draw[thick] (1/27,0.1) -- (1/27,0.7);
	\draw[thick] (2/27,0.35) -- (2/27,0.44);
	\draw[thick] (1/9,0.36) -- (1/9,0.55);
	\draw[thick] (2/9,0.32) -- (2/9,0.4);
	\draw[thick] (7/27,0.4) -- (7/27,0.45);
	\draw[thick] (8/27,0.33) -- (8/27,0.63);
	\draw[thick] (1/3,0.2) -- (1/3,0.5);
	\draw[thick] (2/3,0.35) -- (2/3,0.59);
	\draw[thick] (19/27,0.3) -- (19/27,0.87);
	\draw[thick] (20/27,0.39) -- (20/27,0.46);
	\draw[thick] (7/9,0.43) -- (7/9,0.5);
	\draw[thick] (8/9,0.2) -- (8/9,0.75);
	\draw[thick] (25/27,0.33) -- (25/27,0.37);
	\draw[thick] (26/27,0.31) -- (26/27,0.43);
	\draw[thick] (1,0.35) -- (1,0.4);
	
	\draw[thick] (1/54,0.34) -- (1/54,0.426);
	\draw[thick] (5/54,0.37) -- (5/54,0.43);
	\draw[thick] (13/54,0.345) -- (13/54,0.433);
	\draw[thick] (5/18,0.38) -- (5/18,0.472);
	\draw[thick] (17/54,0.295) -- (17/54,0.377);
	\draw[thick] (37/54,0.325) -- (37/54,0.413);
	\draw[thick] (13/18,0.41) -- (13/18,0.494);
	\draw[thick] (41/54,0.355) -- (41/54,0.447);
	\draw[thick] (49/54,0.34) -- (49/54,0.426);
	\draw[thick] (17/18,0.305) -- (17/18,0.393);
	\draw[thick] (53/54,0.355) -- (53/54,0.439);
	\draw[thick] (3/54,0.37) -- (3/54,0.42);
	
	\fill[black] (0.006,0.305) circle (0.006);
	\fill[black] (-0.007,0.382) circle (0.006);

	\fill[black] (1/54+0.007,0.352) circle (0.006);
	\fill[black] (1/54-0.006,0.401) circle (0.006);

	\fill[black] (1/27+0.008,0.145) circle (0.006);
	\fill[black] (1/27-0.007,0.292) circle (0.006);
	\fill[black] (1/27+0.006,0.516) circle (0.006);
	\fill[black] (1/27-0.008,0.665) circle (0.006);

	\fill[black] (3/54+0.007,0.381) circle (0.006);
	\fill[black] (3/54-0.006,0.411) circle (0.006);

	\fill[black] (2/27+0.007,0.362) circle (0.006);
	\fill[black] (2/27-0.006,0.418) circle (0.006);

	\fill[black] (5/54+0.007,0.381) circle (0.006);
	\fill[black] (5/54-0.007,0.421) circle (0.006);

	\fill[black] (1/9+0.007,0.388) circle (0.006);
	\fill[black] (1/9-0.007,0.471) circle (0.006);
	\fill[black] (1/9+0.006,0.538) circle (0.006);

	\fill[black] (13/54+0.007,0.359) circle (0.006);
	\fill[black] (13/54-0.006,0.421) circle (0.006);

	\fill[black] (2/9+0.006,0.334) circle (0.006);
	\fill[black] (2/9-0.007,0.389) circle (0.006);

	\fill[black] (5/18+0.006,0.396) circle (0.006);
	\fill[black] (5/18-0.007,0.454) circle (0.006);

	\fill[black] (7/27+0.007,0.413) circle (0.006);
	\fill[black] (7/27-0.006,0.442) circle (0.006);

	\fill[black] (17/54+0.006,0.309) circle (0.006);
	\fill[black] (17/54-0.007,0.365) circle (0.006);

	\fill[black] (8/27+0.007,0.347) circle (0.006);
	\fill[black] (8/27-0.007,0.442) circle (0.006);
	\fill[black] (8/27+0.006,0.611) circle (0.006);

	\fill[black] (1/3+0.007,0.228) circle (0.006);
	\fill[black] (1/3-0.006,0.356) circle (0.006);
	\fill[black] (1/3+0.007,0.486) circle (0.006);

	\fill[black] (2/3+0.007,0.368) circle (0.006);
	\fill[black] (2/3-0.006,0.487) circle (0.006);
	\fill[black] (2/3+0.006,0.575) circle (0.006);

	\fill[black] (37/54+0.007,0.339) circle (0.006);
	\fill[black] (37/54-0.006,0.401) circle (0.006);

	\fill[black] (19/27+0.007,0.322) circle (0.006);
	\fill[black] (19/27-0.007,0.458) circle (0.006);
	\fill[black] (19/27+0.006,0.668) circle (0.006);
	\fill[black] (19/27-0.006,0.842) circle (0.006);

	\fill[black] (13/18+0.006,0.424) circle (0.006);
	\fill[black] (13/18-0.007,0.481) circle (0.006);

	\fill[black] (20/27+0.007,0.401) circle (0.006);
	\fill[black] (20/27-0.006,0.451) circle (0.006);

	\fill[black] (41/54+0.007,0.369) circle (0.006);
	\fill[black] (41/54-0.006,0.439) circle (0.006);

	\fill[black] (7/9+0.007,0.441) circle (0.006);
	\fill[black] (7/9-0.006,0.492) circle (0.006);

	\fill[black] (8/9+0.007,0.226) circle (0.006);
	\fill[black] (8/9-0.007,0.389) circle (0.006);
	\fill[black] (8/9+0.006,0.572) circle (0.006);
	\fill[black] (8/9-0.006,0.731) circle (0.006);

	\fill[black] (49/54+0.007,0.353) circle (0.006);
	\fill[black] (49/54-0.006,0.419) circle (0.006);

	\fill[black] (25/27+0.007,0.341) circle (0.006);
	\fill[black] (25/27-0.006,0.366) circle (0.006);

	\fill[black] (17/18+0.006,0.318) circle (0.006);
	\fill[black] (17/18-0.007,0.386) circle (0.006);

	\fill[black] (26/27+0.007,0.327) circle (0.006);
	\fill[black] (26/27-0.006,0.421) circle (0.006);

	\fill[black] (53/54+0.006,0.367) circle (0.006);
	\fill[black] (53/54-0.007,0.433) circle (0.006);

	\fill[black] (1-0.008,0.362) circle (0.006);
	\fill[black] (1-0.006,0.394) circle (0.006);
    \node at (1/2,0) {\small Fra\"iss\'e fence};
	\end{tikzpicture}
    \label{fig:Fraisse}
\end{figure}

A fence over a compact metric space $X$ is determined by a pair $\Phi=(\varphi^L,\varphi^U)$, where $\varphi^L,\varphi^U:X\to[0,1]$ are lower and upper semicontinuous functions with $\varphi^L\leq\varphi^U$, and
\[
\fence=\{(x,t):x\in X,\ \varphi^L(x)\le t\le \varphi^U(x)\}.
\]
Throughout the paper we assume that $X$ is the Cantor space $\cantor$. A \emph{Scissorhand fence} is a fence over $\cantor$ such that the graph of $\varphi^U$ is dense in $\fence$ and the set
\[
\{x\in\cantor:\varphi^L(x)\neq\varphi^U(x)\}
\]
is dense in $\cantor$. If, in addition, the graph of $\varphi^L$ is dense in $\fence$, we call $\fence$ a \emph{two-sided Scissorhand fence}.

Several standard spaces arise as special cases of this construction. For instance, if $\varphi^L\equiv 0$ and $\varphi^U\equiv 1$, then $\fence$ is the \emph{Cantor fence}. Collapsing the Cantor base to a point yields the \emph{Cantor fan}.

For $x\in\cantor$, define the \emph{fiber of $\fence$ at $x$} by
\[
\fence(x):=\{x\}\times\{t\in[0,1]:(x,t)\in\fence\}.
\]

A fence $\fence$ is a Fra\"iss\'e fence if and only if for every $x\in\cantor$ and every continuum $I\subseteq \fence(x)$, there exists a sequence $\{x_n\}\subset \cantor\setminus\{x\}$ with $x_n\to x$ such that each $\fence(x_n)$ is an arc and $\fence(x_n)\to I$ in the Hausdorff metric.

In this terminology, the Fra\"iss\'e fence is a two-sided Scissorhand fence, while the Lelek fence is a Scissorhand fence that is not two-sided.

In the first part of the paper we develop a structural theory of Scissorhand fences. We show that the projections of degenerate components form a dense \(G_\delta\) subset of the Cantor base and, in the two-sided case, the degenerate components themselves form a dense \(G_\delta\) subset of the fence (Proposition~\ref{prop:pointcompdense}). Recall that a map \(f:X\to Y\) is an \emph{almost one-to-one extension} if the set \(\{\, y\in Y : f^{-1}(y)\text{ is a singleton} \,\}\) is dense in \(Y\). As a consequence, Scissorhand fences are almost one-to-one extensions of the Cantor space.

We also provide an inverse limit construction of fences defined by semicontinuous mappings over the Cantor space. This construction yields precise control over the fibers and their density properties and leads to a characterization of the fences considered here (Theorem~\ref{thm:fenceconst}). Within this framework, we recover standard examples such as the Lelek fence (Example~\ref{ex:LeLekFrA}) and the Fra\"iss\'e fence (Example~\ref{ex:FraisseFence}), as well as new examples of two-sided Scissorhand fences that are not Fra\"iss\'e fences (Example~\ref{ex:bifusednotfraisse}).

\subsection{Dynamics on spaces}

Some of the spaces described above have also been studied from a dynamical perspective. Aarts and Oversteegen \cite{AartsOversteegen95} showed that the homeomorphism group of the hairy arc is one-dimensional and totally disconnected. Barto\v{s}ov\'a and Kwiatkowska \cite{BartosovaKwiatkowska2015} identified the universal minimal flow of the homeomorphism group of the Lelek fan as the natural action on the compact space of maximal chains of subcontinua containing the top point.

Important connections with complex dynamics arise from the study of escaping sets of transcendental entire functions. Eremenko \cite{Eremenko1989} conjectured that every component of the escaping set
\[
I(f)=\{z\in\mathbb C : f^n(z)\to\infty\}
\]
is unbounded. Rottenfußer, Rückert, Rempe, and Schleicher \cite{RottenfusserRuckertRempeSchleicher} established strong forms of this conjecture for broad classes of functions of bounded type. In particular, for functions of finite order, the escaping set consists of injective curves, called \emph{dynamic rays}, each tending to infinity, whose union (together with possible endpoints) forms a Cantor bouquet in the sense of Aarts and Oversteegen \cite{AO93}. 

Barański, Jarque, and Rempe \cite{BaranskiJarqueRempe2012} showed that the Julia set of any bounded-type finite-order transcendental entire function contains a Cantor bouquet and that, in the disjoint-type case, the entire Julia set has this structure. More recently, Pardo-Sim\'on and Rempe \cite{MashaelRempeSixsmith2022} proved that, within the disjoint-type class, a transcendental entire function has a Julia set homeomorphic to a \emph{Cantor bouquet} if and only if it is \emph{criniferous}, i.e.\ every escaping point eventually lies on a dynamic ray. Under mild geometric assumptions, they also showed that the \emph{head-start condition}, previously known to be sufficient, is in fact necessary for the Julia set to be a Cantor bouquet.

Minimal dynamical systems on spaces with arc-like fibers have been studied extensively. Floyd \cite{Floyd} constructed a non-homogeneous minimal extension of an odometer, now known as a Floyd fence. Auslander \cite{Auslander1959} introduced a minimal mean-$L$-stable but non-distal system projecting onto the triadic adding machine, later generalized by Haddad and Johnson \cite{HaddadJohnson97}. Deeley, Putnam, and Strung \cite{DeeleyPutnamStrung} further developed constructions of minimal extensions with controlled fiber structure. 

Balibrea, Downarowicz, Hric, Snoha, and \v{S}pitalsk\'y \cite{BalibreaDownarowiczHricSnohaSpitalsky} showed that minimal Cantor systems admit extensions to minimal \emph{non-invertible} maps on almost totally disconnected spaces (cantoroids), a class that includes the Fra\"iss\'e fence but not the Lelek fence. The structure of Floyd--Auslander systems was later analyzed in detail by V\'ybo\v{s}tok \cite{Vybostok}.

Recent work in topological dynamics shows that the Lelek fan supports rich dynamical behavior. Banič, Erceg, Kennedy, Mouron, Nall, and Jelić constructed transitive and mixing homeomorphisms on the Lelek fan and related smooth fans \cite{banic2023transitive,banic2024chaos,banic2025fraisse}. Oprocha \cite{Oprocha2024} constructed a completely scrambled weakly mixing homeomorphism of the Lelek fan.

Our goal is to develop a systematic framework for lifting dynamics from the Cantor space $\cantor$ to Scissorhand fences $\fence$. As a first step (Section~\ref{sec:Dynamics that directly lifts}), we show that certain dynamical properties can be deduced without explicitly constructing the lifting. In particular, for two-sided Scissorhand fences, if a continuous surjection on the Cantor base admits a lifting, then transitivity and minimality are preserved, and, in the case of homeomorphisms, chaotic behavior is also inherited (Theorem~\ref{thm:GenFraisse}). Moreover, combining results of Bowen \cite{Bowen} and Kolyada and Snoha \cite{KolyadaSnoha}, we show that when the fiber dynamics is given by homeomorphisms, the lifted system preserves topological entropy (Proposition~\ref{prop:entropy}).

In Section~\ref{sec:General theorem} we prove a realization theorem (Theorem~\ref{thm:mainF}) providing a systematic method for lifting maps from the Cantor space to Scissorhand fences. Using inverse limit representations of Cantor systems \cite{AkinGlasnerWeiss,DarjiBernardes,shimomuraGraphCover}, we introduce $\cF$-systems encoding both the geometry of the fence and the induced dynamics (Definition~\ref{def:fsystem}), and producing a fence $\fence$ together with a map on $\fence$. 
More precisely, under Condition~$\Gamma$ (see \eqref{eq:gamma}), an $\cF$-system determines a uniquely defined continuous surjection on $\fence$ with the original Cantor map as a factor, preserving topological entropy. The construction also preserves key structural properties: homeomorphisms lift to homeomorphisms, Lipschitz and bi-Lipschitz regularity are retained under the corresponding assumptions, and, in the case of unit scaling, isometries lift to isometries.

The subsequent sections illustrate applications of this framework in concrete settings. We focus on the Lelek fence, which is homeomorphic to the hairy Cantor set \cite{haircantorset} and the hairy arc \cite{AO93}, and on the Fra\"iss\'e fence, as these are precisely the cases where a topological characterization allows us to identify the resulting spaces. These examples are not exhaustive, and further applications are possible.

In Section~\ref{sec: Applications to dynamics on isometries} we apply the realization theorem to Cantor isometries with nowhere dense orbits. We show that such systems admit liftings to both the Fra\"iss\'e fence (Theorem~\ref{thm:liftingisometriesFraisse}) and the Lelek fence (Theorem~\ref{thm:liftingisometriesLelek}) that remain isometries and preserve the factor structure. In the Fra\"iss\'e case, the construction allows one to prescribe that the lower and upper endpoint functions coincide on a given union of invariant nowhere dense sets, while in the Lelek case the upper endpoint function can be chosen strictly positive on such sets. 
As a consequence (see Remark~\ref{rem:isometries}), periodic points of the lifted systems are localized along fibers over periodic points of the base in the Lelek case, and on degenerate components in the Fra\"iss\'e case.

In Section~\ref{sec: Applications to dynamics on Lelek fence} we study the lifting of specific dynamical properties from Cantor systems to the Lelek fence. We show that transitivity can be lifted so that a prescribed upper endpoint $(x,\varphi^U(x))$ is a transitive point of the lifted system (Theorem~\ref{thm:transitive}). We then consider chaotic dynamics, proving that every chaotic Cantor homeomorphism admits a chaotic lifting to the Lelek fence that preserves the factor relation (Theorem~\ref{thm:chaotic}). Finally, we address topological mixing: under an additional recurrence condition involving invariant nowhere dense subsets, a broad class of topologically mixing Cantor homeomorphisms admits liftings that are themselves topologically mixing on the Lelek fence (Theorem~\ref{thm:mixing}). In particular, this applies to shift homeomorphisms (see Example~\ref{rmk:mixing}). By collapsing the Cantor base to a point, the statements of Theorems~\ref{thm:liftingisometriesLelek}, \ref{thm:transitive}, \ref{thm:chaotic}, and \ref{thm:mixing} also hold for the Lelek fan, which is homeomorphic to the Cantor bouquet.

In Section~\ref{sec: Applications to dynamics on Fraisse fence} we generalize the realization method from Section~\ref{sec:General theorem} to obtain finer control over the resulting dynamics, applicable also to two-sided Scissorhand fences (Theorem~\ref{thm:main}). As an application, we show that odometer Cantor systems admit liftings to minimal homeomorphisms of the Fra\"iss\'e fence (Theorem~\ref{thm:minimalfraisse}), providing the first such examples on this space. Moreover, this construction yields uncountably many pairwise non-conjugate minimal homeomorphisms on the Fra\"iss\'e fence, none of which factors onto another (Corollary~\ref{cor:Fraisse minimal}).

\section{Fences}
\subsection{Definition of fences and fans}
A \emph{Cantor space} is a compact metric space, with a countable basis of clopen sets and having no isolated points. By Brouwer's theorem, up to homeomorphism there is only one such space and they are all homeomorphic to the standard middle 1/3 Cantor set on the real line. There are various models of Cantor space. We use the one that is most suitable in a given context. We use $\cantor$ to denote a \emph{Cantor space}.

\begin{definition}
For a pair of functions  $\Phi=(\varphi^L,\varphi^U)$ where $\varphi^L,\varphi^U: X \rightarrow [0,1]$, $\varphi^{L}\leq \varphi^U$ and $\varphi^{L}, \varphi^U$ are lower  and upper semicontinuous, respectively,   we let
\[ \fence = \{(x,t): x \in X, \varphi^L(x) \le t \le \varphi^U(x) \}. \] 
 ${\bf F}_{\Phi}$ is a compact subspace of $X \times [0,1]$ and will be called a \emph{fence over $X$}.
\end{definition}

 Our definition of fences is concrete as opposed to the ones defined in \cite{BassoCamerlo}. Our  definition of fences forms a subcollection of smooth fences in \cite{BassoCamerlo}. Our main objective is to construct homeomorphisms with interesting dynamical properties on well-studied fences and fans. Henceforth, we will assume that $X$ is the Cantor space $\cantor$.

There are already well-studied examples of such fences in the literature. For example, when $\varphi^L=0$ and $\varphi^U=1$ then $\fence$ is the \emph{Cantor fence}. If we quotient the base Cantor space to a point we obtain the \emph{Cantor fan}.


Motivated by Complex Dynamics, an object called the Hairy Cantor set was extensively studied in \cite{haircantorset}. 
\emph{Hairy Cantor set}, a set in the plane, was defined axiomatically \cite{haircantorset} as a certain type of a compact subset of the plane and it was shown that any two such Hairy Cantor sets are ambiently homeomorphic. 
It turns out (using  Corollary 5.6 in \cite{haircantorset} or arguments similar to that of \cite{oversteegenOneLF}) that Hairy Cantor set is homeomorphic to the set
\[ \ffence{\varphi}{} = \{(x,t): x \in C, 0 \le t \le \varphi(x) \} \]
where $C$ is a standard Cantor set and $\varphi: C \rightarrow [0,1]$ is an upper semicontinuous function such that $\varphi$ is zero on a dense set, positive on a dense set and the graph of $\varphi$ is dense in $\ffence{\varphi}{}$. 
We work with the above   abstract model of the Hairy Cantor set and call it the Lelek fence.

\begin{definition}
A \emph{Lelek fence} is a fence $\fence$ where $\Phi=(0,\varphi^U)$ and $\varphi^U$ is positive on a dense set and the graph of $\varphi^U$ is dense in $\fence$. 
\end{definition}

\begin{remark}
 If  $\fence$ is a Lelek Fence,  then the  function $\varphi^U$ is zero on a dense $G_\delta$ set. Indeed, for each $r >0$, set $U_r = \{x: \varphi^U  (x) < r \}$ is open as $\varphi^U$ is upper semi-continuous and dense in $\cantor$ as the graph of $\varphi^U$ is dense in $\fence$. Hence, $\cap_{r > 0} U_r$ is a dense $G_{\delta}$ set where  $\varphi^U$ is zero.
\end{remark}

As discussed earlier, Lelek fan, a cousin of Lelek fence, has been extensively studied from various perspectives in Topological Dynamics. 
When the base of Lelek fence is identified to a point, we obtain the Lelek fan, a compact connected set with properties similar to that of Lelek fence. Namely, \emph{Lelek fan} is a set homeomorphic to the quotient $\ffence{\varphi}{}/\mathcal{E}$ where $\ffence{\varphi}{}$ is a Hairy Cantor set and 
\[ \mathcal{E} = \{ ((x,0), (y,0)) \in \ffence{\varphi}{}^2: x, y \in \cantor \} \cup \{ ((x,s),(x,s)) \in \ffence{\varphi}{}^2 : x, y \in \cantor \}.\]
All Lelek fans are homeomorphic to each other \cite{oversteegenOneLF}. 

Fra\"iss\'e fence \cite{BassoCamerlo}, an interesting object arising from descriptive set-theoretic studies of projective Fra\"iss\'e limit, was introduced by Basso and Camerlo. Numerous properties of Fra\"iss\'e fence were proved in \cite{BassoCamerlo}, including uniqueness and certain types of homogeneity and universality. 
A topological characterization of Fra\"iss\'e fence  was given in \cite[Theorem 5.3]{BassoCamerlo}. Motivated by this characterization, we give a concrete, geometric formulation of Fra\"iss\'e fence as below. For an $x\in \cantor$ let

\[
\fence (x) := \{t \in [0,1]: (x,t) \in \fence\}.
\]

\begin{definition}\label{def:fraisse}
    A fence $\fence$ is a Fra\"iss\'e fence if and only if 
 for each $x \in \cantor$ and each continuum $I \subseteq \fence (x)$, there is a sequence $\{x_n\}$ in $\cantor \setminus \{x\}$ converging to $x$ such that each $\fence (x_n)$ is an arc, and $\{\fence (x_n)\}_{n \in \N}$ converge to $I$ in the Hausdorff metric. 
\end{definition}

Using the definition of Hausdorff metric one can easily verify that our definition of Fra\"iss\'e fence agrees with the following formulation of Fra\"iss\'e fence given in {\cite[Theorem 5.3]{BassoCamerlo}.  
\begin{proposition}\label{prop:denseG_delta}
Let $\Phi=(\varphi^L,\varphi^U)$. A fence $\fence$ is a Fra\"iss\'e fence if and only if for any two open sets $O,O'\subset {\bf F}_{\Phi}$ which meet a common connected component of ${\bf F}_{\Phi}$, there is an arc component $A$ of  ${\bf F}_{\Phi}$ such that one endpoint of $A$ 
belongs to $O$ and the other endpoint of $A$ belongs to $O'$.  
\end{proposition}

We now introduce a natural larger class of fences which includes Fra\"iss\'e fence and Lelek Fence. 
These fences can be thought of as  floating Lelek fences, i.e., arc components roam freely in the fence in a dense way. 

\begin{definition}\label{def:scissorhand}
A \emph{Scissorhand Fence} ({\bf SF}) is  $\fence$, where $\Phi=(\varphi^L,\varphi^U)$, such that
the graph of $\varphi^U$ is dense in $\fence$, and 
$\{x \in \cantor: \varphi^L(x) \neq \varphi^U (x) \} $ is dense in $\cantor$. Moreover, if, in addition, the graph of $\varphi^L$ is dense in $\fence$, we call $\fence$ \emph{two-sided Scissorhand Fence} ({\bf TSF}).
\end{definition}

It is clear that every Fra\"iss\'e fence is a two-sided Scissorhand fence. In Example~\ref{ex:bifusednotfraisse} we will show that there exists a  two-sided Scissorhand fence which is not homeomorphic to the Fra\"iss\'e fence. Hence, this class of fences is more general. The following proposition illuminates further property of Scissorhand fences and two-sided Scissorhand fences. For the case of Fra\"iss\'e fence, the following proposition was proven in \cite[Proposition 5.19]{BassoCamerlo}. 

\begin{proposition}\label{prop:pointcompdense} Let $\fence$ be a fence. 
\begin{enumerate}

\item If $\fence$ is a  Scissorhand fence, then \[ {\bf D}_1 =\left \{x \in \cantor: \{y\} = \fence(x)\right \} \] is dense $G_{\delta}$ in $\cantor$.
    \item If $\fence$ is a two-sided Scissorhand fence, then the set \[  {\bf D}_2 = \left \{(x,y) \in \fence: \{y\} = \fence(x)\right \} \] is dense $G_{\delta}$ in $\fence.$
\end{enumerate}
\end{proposition}
\begin{proof}
Let $\fence$ be a fence.
For each $r>0$, let 
    \[U_{r}=\{(x,y) \in \cantor \times [0,1]: y \in \fence(x),\  diam (\fence(x)) < r \}.
    \]
    First we show that $U_r$ is open. Indeed, let $(x,y) \in U_r$, $a = \varphi^L(x)$, $b= \varphi^{U}(x)$, $a_1<a$, $b_1>b$ such that $b_1-a_1<r$. By the fact that $\varphi^U,\varphi^L$ are upper and lower semi-continuous, respectively, we can find an open set $O$ in $\cantor$ containing $x$ such that for all $x' \in O$, we have that $\varphi^U (x)<b_1$, and $\varphi^L(x) >a_1$. Then, $[O \times (a_1, b_1)] \cap \fence$ is an open set containing $(x,y)$ and a subset of $U_r$.
 
    Now assume that $\fence$ is a Scissorhand fence. We next show that given $(x, \varphi^L(x))$ and $\varepsilon >0$, there is a point of $U_r$ within $\varepsilon$ of $(x,\varphi^L(x))$. We may assume that $\varepsilon < r/2$. As $\varphi^L$ is lower semi-continuous, there is an open neighborhood $O$  of $x$ with diameter less than $\varepsilon$ such that  for all $t \in O$, we have that $\varphi^L(t) > \varphi^L(x) - \varepsilon$. Let $(a,b)$ be  a neighborhood of $\varphi^L(x)$ in $[0,1]$ with diameter less than $\varepsilon$.  As the graph of $\{(t,\varphi^U(t)): t \in \cantor\}$ is dense in $\fence$, there is $(x',y') \in [O \times (a,b)] \cap \fence$ such that $y' = \varphi^U(x')$. As $x' \in O$, we have that $\varphi^L(x') >\varphi^L(x) - \varepsilon$, implying that $\varphi^U(x')- \varphi^L(x')< 2 \varepsilon < r$ and $(x',y') \in U_r$.

    From above we have that $ \pi_1(U_r)$, the projection of $U_r$ onto the first coordinate, is open and dense in $\cantor$. Moreover, ${\bf D_1} = \cap_{r=1}^{\infty} \pi_1(U_r)$, a dense, $G_{\delta}$ subset of $\cantor$.

    To see (2) from the observation above and the fact that ${(x,\varphi^L(x)): x \in \cantor}$ is dense in $\fence$, we have that $U_r$ is dense  and open in $\fence$. As ${\bf D_2} = \cap_{r=1}^{\infty} U_r$, we have that ${\bf D_2}$ is a dense $G_{\delta}$ subset of $\fence$.

\end{proof}

We now introduce some techniques for constructing variety of fences using inverse limit spaces. These techniques will be expanded later to construct maps on fences with various dynamics.

\subsection{Inverse limit construction of fences}
Cantor space $\cantor$ can also be constructed as an inverse limit space. The following definition captures this.
\begin{definition}\label{def:cstructure}
    A \emph{$\cC$-structure} is a sequence $\cstructure$ where $V_n$ is a finite set and $\Psi_n:V_{n+1} \rightarrow V_n$ is a surjective map satisfying the following condition:
    \begin{itemize}
        \item for each $n \in \N$ and $v \in V_n$, there exists $m >n$ and $v' \neq v'' \in V_{m+1}$ such that $\Psi_n^m (v') = \Psi_n^m(v'') = v$ (here, $\Psi^m_n:=\Psi_n\circ\ldots \circ \Psi_m$).
    \end{itemize}
    The inverse limit $X =\underleftarrow{\lim} (V_n, \Psi_n)$ is a Cantor space.
    If $x \in X$ and $n \in \N$, then $x(n)$ is the $n$-th coordinate of $x$. Moreover, for $v \in V_n$, we let $[v] = \{x\in X: x(n) = v\}$.
\end{definition}

Based on $\cC$-structure we define $\cF$-structure which yields a general inverse limit type construction of fences.

\begin{definition}\label{def:fstructure} An \emph{ $\cF$-structure} is a sequence $\fstructure$ consisting of $\cC$-structure $\cstructure$, and mappings $\{\varphi^U_n:V_n\to [0,1]\}$ and $\{\varphi^L_n:V_n \to [0,1]\}$ satisfying the following conditions.
    \begin{enumerate}
    \item $\varphi^{L}_n(v) \le \varphi^U_n(v)$ for all $n\in \N$ and $v\in V_n$,
    \item $\varphi^U_{n+1}(v')\leq \varphi^U_{n}(v)$ and $\varphi^L_{n+1}(v')\geq \varphi^L_{n}(v)$  whenever $\Psi_n(v')=v$.
\end{enumerate}
Let $\varphi^U$ and $\varphi^L$ be the limit of $\{\varphi^U_n\}$ and $\{\varphi^L_n\}$, respectively. Then, $\varphi^U$ and $\varphi^L$ are upper and lower semicontinuous, respectively and $\fence$ is a fence, where $\Phi=(\varphi^L,\varphi^U)$.
\end{definition}

Next theorem gives conditions which characterizes certain types of fences. However, first we  define some parameters.
\begin{definition}\label{def:etaparameters}

For $n \in \N$, $g \in G_n$ and an interval $I \subseteq [\varphi_{n}^L (g), \varphi_{n}^U (g)]$, let 

\[ \eta(g, I)= \inf \left \{d_H \left (I, [\varphi_{n+1}^L (g'), \varphi_{n+1}^U (g')] \right ):  \ g' \in G_{n+1}, \  \Psi_n(g') = g  \right \}.
\]
For $n \in \N$, $g \in G_n$, $t \in [\varphi_{n}^L (g), \varphi_{n}^U (g)]$, let

\[ \eta^+(g,t)= \inf \left \{ |t-\varphi_{n+1}^U (g')|:  \ g' \in G_{n+1}, \  \Psi_n(g') = g  \right \}
\]

\[ \eta^-(g,t)= \inf \left \{ |t-\varphi_{n+1}^L (g')|:  \ g' \in G_{n+1}, \  \Psi_n(g') = g  \right \}
\]

Let 

\begin{equation}\label{eq:eta_n}
\eta_n = \max \{\eta(g, I): g \in 
G_n, I \subseteq [\varphi_{n}^L (g), \varphi_{n}^U (g)]\},  
\end{equation}
\begin{equation}\label{eq:eta_n+}
\eta^+_n = \max \{ \eta^+(g, t ): g \in 
G_n,  \  t \in [\varphi_{n}^L (g), \varphi_{n}^U (g)]\}  
\end{equation}
\begin{equation}\label{eq:eta_n-}
\eta^-_n = \max \{\eta^-(g, t ): g \in 
G_n\ t \in [\varphi_{n}^L (g), \varphi_{n}^U (g)]\} .
\end{equation}
Note that $\eta_n \ge \eta^+_n, \eta^-_n$.
\end{definition}

The following theorem serves as a key tool in our constructions for identifying a particular class of fences.

\begin{theorem}\label{thm:fenceconst} Let $\fstructure$ be an $\cF$-structure and $\fence$ be its associated fence. Furthermore, assume that the following condition is satisfied.  

\begin{equation}\tag{$\dagger$}\label{eq:dagger}
\forall v \in V_n,\  \exists v' \in V_{n+1} \text{ such that } \Psi_n(v') = v , \  \varphi^L_n(v) = \varphi^L_n(v') \text{ and } \varphi^U_n(v) = \varphi^U_n(v').\
\end{equation}

    \begin{enumerate}
    \item If $\varphi^L$ is the zero function and $\varphi^U$ is constant one function, then $\fence$ is the Cantor fence.
    \item  If $\{\eta^+_n\}  \rightarrow 0$, then $\fence$ is the Scissorhand Fence.
    \item If $\varphi^L$ is the zero function and $\{\eta^+_n\} \rightarrow 0$, then $\fence$ is the Lelek fence.
    \item If $\{\eta^+_n\} \rightarrow 0$, and $\{\eta^-_n\} \rightarrow 0$ then $\fence$ is a two-sided Scissorhand Fence.
        \item If $\{\eta_n\} \rightarrow 0$, then $\fence$ is the Fra\"iss\'e fence.

    
    \end{enumerate}
    
\end{theorem}

\begin{proof}
Part 1. of the theorem is the definition of the Cantor fence.

Note that $\fence = \cap _{n=1}^{\infty} {\bf F}_{(\varphi^L_n, \varphi^U_n)}$.
Moreover, the condition \eqref{eq:dagger} implies for $ g\in G_n$, there exists $x_1, x_2 \in [g]$ such that $\varphi^L(x_1) = \varphi_n^L(x_1)$ and $\varphi^U(x_2) = \varphi_n^U(x_2)$. In particular, we have that $\{x \in \cantor: \varphi^L(x) \neq \varphi^U (x) \} $ is dense in $\cantor$.

Now to see Part 2., we only need to verify that the graph of $\varphi^U$ is dense in $\fence$. Indeed, this follows from the definition of $\eta^+_n$ and the fact that for all $g \in G_n$, there exists $x \in [g]$ such that $\varphi^U(x) = \varphi_n^U(x)$.

Parts 3.  and 4. are analogous, we simply use the definitions of $\eta^+_n$ and $\eta^-_n$.

Part 5. follows from the definition of $\eta_n$.
\end{proof}

\begin{example}\label{ex:bifusednotfraisse}
    There exists a two-sided Scissorhand Fence which is not the Fra\"iss\'e fence.
   \end{example} 
\begin{proof}
    This simply follows from the fact that one can do the above construction where Condition 4 of Theorem~\ref{thm:fenceconst} holds but Condition 5 does not. Indeed, if one constructs a sequence of clopen partitions $\{G_n\}$ of $\cantor$,  $G_{n+1}$ refining $G_n$ such that  
the following holds 
\begin{enumerate}[label=(\alph*)]
 \item if $g' \in G_{n+1}$, $g \in G_n$ such that $g' \subseteq g$, then either $\varphi^U_{n+1} (g') = \varphi^U_{n} (g) $ or $\varphi^L_{n+1} (g') = \varphi^L_{n} (g) $,
  
    \item if $g' \in G_{n+1}$, then $|\varphi_{n+1}^L(g') -\varphi_{n+1}^U(g')| = \ell \cdot |\varphi_{n}^L(g) -\varphi_{n}^U(g)|$ where $\ell \in \{1/2, 1 \}$,
\item for all $g \in G_n$, there is $g',g_L, g _U \in G_{n+1}$ with $g',g_L, g_U  \subseteq g$ such that 
    \begin{itemize}
        \item $|\varphi_{n+1}^L(g') -\varphi_{n+1}^U(g')| = |\varphi_{n}^L(g) -\varphi_{n}^U(g)|$,
        \item $\varphi^U_{n+1} (g_U) = \varphi^U_{n} (g) $ and  $|\varphi_{n+1}^L(g_U) -\varphi_{n+1}^U(g_U)| = 1/2 \cdot |\varphi_{n}^L(g) -\varphi_{n}^U(g)|$,
        \item $\varphi^L_{n+1} (g_L) = \varphi^L_{n} (g) $ and  $|\varphi_{n+1}^L(g_L) -\varphi_{n+1}^U(g_L)| = 1/2 \cdot |\varphi_{n}^L(g) -\varphi_{n}^U(g)|$,
    \end{itemize}
\end{enumerate}
 then, by Theorem~\ref{thm:fenceconst}, the resulting fence is a two-sided Scissorhand Fence because $\eta^+_{n}=2^{-n}=\eta^-_{n}$. That the resulting fence is not a Fra\"iss\'e Fence  simply follows from the definition of Fra\"iss\'e fence (Definition~\ref{def:fraisse}).
 \end{proof}

 The fences constructed in Example~\ref{ex:bifusednotfraisse} are homeomorphic to the underlying spaces of Auslander systems constructed explicitly by \cite{HaddadJohnson97}.

We next show how to modify the above construction to obtain the Lelek Fence.

\begin{example}\label{ex:LeLekFrA} (Lelek Fence) 
We construct a sequence of clopen partitions $\{G_n\}$ of $\cantor$,  $G_{n+1}$ refining $G_n$ such that  
the following holds. As usual, we initialize by $\varphi^U_0(g)=1$, and $\varphi^L_0 (g)=0$ for $g \in G_0$. For $n \in N$, we require that  
   \begin{enumerate}

  \item $\varphi^L_{n+1} (g')= \varphi^L_{n}(g)$ where $g' \subseteq g$, 
    \item for every $g\in G_n$ and for every $\ell \in \{2^{-n}, 2\cdot 2^{-n},\ldots, 2^n\cdot 2^{-n}=1\}$ there is $g' \in G_{n+1}$ such that 
    
    \[\varphi_{n+1}^U(g') = \varphi_n^L(g) + \ell \cdot (\varphi_{n}^U(g) - \varphi^L_n(g)).\]
    \end{enumerate}
   As $\eta^+_{n}\le 2^{-n}$,  by Theorem~\ref{thm:fenceconst} (3.) the resulting fence is a Lelek Fence. Note that $\varphi^L_n(g)=0$ for all $g$. However, we have written as above to facilitate the construction of Fra\"iss\'e Fence below.
\end{example}
Next we slightly modify the construction of Lelek Fence at even and odd steps so the resulting fence is a Fra\"iss\'e Fence.
\begin{example}\label{ex:FraisseFence}(Fra\"iss\'e Fence)
 As usual, we initialize by $\varphi^U_0(g)=1$, and $\varphi^L_0 (g)=0$ for $g \in G_0$.  Suppose $G_n, \varphi^U_n$, $\varphi^L_n$ has been constructed. We construct $G_{n+1}$ and $\varphi^U_n$, $\varphi^L_n$ in two steps. We first mimic Lelek fence construction from Example~\ref{ex:LeLekFrA}. Then this intermediate step will be modified by a dual construction of the Lelek fence.\\ 
{\bf Step 1.} Choose $\widehat{G}_{n+1}, \widehat{\varphi}^U_{n+1}$ and $\widehat{\varphi}^L_{n+1}$ so that 
 
 \begin{enumerate}
  \item $\widehat{\varphi}^L_{n+1} (g')= \varphi^L_{n}(g)$ where $g' \subseteq g$, 
    \item for every $g\in G_n$ and for every $\ell \in \{2^{-n}, 2\cdot 2^{-n},\ldots, 2^n\cdot 2^{-n}=1\}$ there is $g' \in \widehat{G}_{n+1}$ such that 
    
    \[\widehat{\varphi}_{n+1}^U(g') = \varphi_n^L(g) + \ell \cdot (\varphi_{n}^U(g) - \varphi^L_n(g)).\]
    \end{enumerate}
    {\bf Step 2.} Next we modify the intermediate stage  by exchanging the role of $U$ and $L$ in the  construction  of  Example~\ref{ex:LeLekFrA}.  Namely, we choose $G_{n+1}$, $\varphi^U_{n+1}$ and $\varphi^L_{n+1}$ so that
   \begin{enumerate}
  \item  $\varphi^U_{n+1} (g')= \widehat{\varphi}^U_{n+1}(g)$  where $g' \subseteq g$, 
    \item for every $g\in \widehat{G}_{n+1}$ and for every $\ell \in \{2^{-n}, 2\cdot 2^{-n},\ldots, 2^n\cdot 2^{-n}=1\}$ there is $g' \in G_{n+1}$ such that 
    \begin{align*}
    \varphi_{n+1}^L(g') = \widehat{\varphi}_{n+1}^U(g) - \ell \cdot (\widehat{\varphi}_{n+1}^U(g) - \widehat{\varphi}^L_{n+1}(g)).
        \end{align*}
    \end{enumerate}
    The first step guarantees us that $\eta^+_{n+1} \leq 1/2^n$. Combining it with the second step, we have that $\eta_n\leq 1/2^n$. As $\eta_n\leq 1/2^n$,
   applying Theorem~\ref{thm:fenceconst} (5.) we have that the resulting fence is a Fra\"iss\'e Fence. 
\end{example}

\section{Dynamics that directly lifts on fences}\label{sec:Dynamics that directly lifts}

This section is devoted to fundamental properties of lifting dynamics from the base Cantor space $\cantor$ to fences $\fence$. We prove that under homeomorphic fiber dynamics, the lifted system has the same topological entropy as the base system (Proposition~\ref{prop:entropy}). Furthermore, we prove that, for two-sided Scissorhand fences, liftings preserve transitivity, minimality, and, for homeomorphisms, chaotic behavior (Theorem~\ref{thm:GenFraisse}). 

\begin{definition}\label{def:lifting} 
    Let $\homeocantor: X \rightarrow X$ and $\fence$ be a fence over $X$. We say that $T: \fence \rightarrow \fence$ is a lifting of $\homeocantor$, if $T(x,t) = (\homeocantor(x),s)$ for some $s \in \fence (\homeocantor(x))$.
\end{definition}
The following theorem holds for arbitrary fence over Cantor space.  
\begin{proposition}\label{prop:entropy}
    Suppose $\fence$ is a fence and $T:\fence \rightarrow \fence$ lifts a continuous surjection $\homeocantor:X \rightarrow X$ in a way so that the restriction of $T$ to $\fence(x)$ is a homeomorphism for each $x \in X$. Then, $\ent(T) = \ent (\homeocantor)$.
\end{proposition}

\begin{proof}
Recall that by \cite[Theorem 17]{Bowen}

\[\ent(\homeocantor)\leq \ent(T)\leq \ent(\homeocantor) + \sup_{x\in X} \ent(T|_{ \fence(x)}). \]
It can be shown, applying \cite[Theorem D]{KolyadaSnoha} and some additional technical details, that $\ent(T|_{\fence(x)})=0$ for every $x\in X$.
  For the sake of completeness, we give a direct proof of this fact. Indeed, for $\varepsilon>0$ any subset of $\fence(x)$ that is $\varepsilon$-separated has cardinality at most $\frac{1}{\varepsilon}$. Hence, for $n \in \N$, denote by $A(n, \varepsilon,x)$ the set of $(n,\varepsilon)$-separated subset of $\fence(x)$. The set $A(n, \varepsilon,x)$ has cardinality at most $\frac{1}{\varepsilon}\cdot n$. This can be seen by induction and the fact that $\{t_1 < t_2 < \ldots < t_j\} \subseteq \fence(x)$ are $(n,\varepsilon)$-separated if and only if $\{t_i,t_{i+1}\}$ are $(n,\varepsilon)$-separated for every $1 \le i < j$. This is indeed the case as $T|_{\fence(x)}$ is a homeomorphism.  Now, from Bowen's definition of entropy, it follows that  

\[ \ent(T|_{\fence(x)})= \lim_{\varepsilon\to 0}\limsup_{n\to\infty} \frac{A(n,\varepsilon,x)}{n}\leq \lim_{\varepsilon\to 0}\limsup_{n\to\infty} \frac{\log\left(\frac 1{\varepsilon}n\right)}{n}=0.\]
Hence, $\ent(T) = \ent (\homeocantor)$.

\end{proof}
\begin{theorem}\label{thm:GenFraisse}
Suppose $\fence$ is a two-sided Scissorhand fence and $T:\fence \rightarrow \fence$ lifts a continuous surjection $\homeocantor:X \rightarrow X$.

 \begin{enumerate}
 \item If $\homeocantor$ is transitive then $T$ is transitive. 
 \item If $\homeocantor$ is minimal, then $T$ is minimal.
 \item If $T$ is a homeomorphism and  $\homeocantor$ is chaotic, then $T$ is chaotic.
 \end{enumerate}
\end{theorem}

\begin{proof} Let us first note that since $\fence$ is a two-sided Scissorhand Fence, by Proposition~\ref{prop:pointcompdense}(2), the set of points, which are contained in degenerate components, is dense in $\fence$. Moreover, using this fact and the nature of $\fence$, we have that every nonempty open set in $\fence$ contains a set of the form ${\mathbf F}_{\Phi}|_{U}$ for some nonempty open set $U$ subset of  $X$.

For (1) we will verify that $(x,t) \in \fence$ is a  transitive point of $T$ whenever $x$ is a transitive point of $\homeocantor$. To this end, let $O$ be a nonempty open set in $\fence$. By our observation above, there is nonempty open $U$ in $X$ such that  ${\mathbf F}_{\Phi}|_{U} \subseteq O$.  As $x$ is a  transitive point of $\homeocantor$,  there is $n \in \N$ such that $\homeocantor^n(x) \in U$. Then, $T^n(x,t) \in \fence(\homeocantor^n(x)) \subseteq  \fence|_{U} \subseteq O$.

Part (2) follows the proof of Part (1) and the fact that every point of $X$ is a transitive point of $\homeocantor$. 

For (3), in light of (1), we only need to verify that the set of periodic points of $T$ is dense in $\fence$. Let $O$ be an open set in $\fence$. By our observation above, there is nonempty open $U$ in $X$ such that  $\fence|_{U} \subseteq O$. As $\homeocantor$ is chaotic, we may choose a point $x \in U$ which is a periodic point of $\homeocantor$.  As the  endpoints of $\fence(x)$ map under $T$ to  endpoints of $\fence(y)$ for some $y \in X$, we have that for $t$ an endpoint of $\fence(x)$, $(x,t)$ is a periodic point of $T$  of the same period as the period of $x$ under $\homeocantor$ or twice the period of period of $x$ under $\homeocantor$.
\end{proof}

\section{Maps on fences}\label{sec: Maps on fences}

In this section we prepare the groundwork for the proof of the main results by introducing a framework for lifting maps from the Cantor space to fences. We define $\cC$-systems, which encode dynamical systems on the Cantor space, and $\cF$-systems, which describe corresponding dynamical systems on fences. Within this framework, we will in subsequent section establish a general result that allows one to lift maps on the Cantor space to maps on fences.

\subsection{ $\cC$-systems}\label{subsec:csystmes}
Throughout, a \emph{digraph} is a directed graph $G= (V,E)$ with vertex set $V$ and set $E$ consisting of directed edges. Furthermore, we will assume that each vertex has at least one outgoing edge and at least one incoming edge. To expedite notation, we will usually use $G$  as the set of vertices and will use notation $\oa{uv} \in G$ to indicate that $\oa{uv}$ is an edge in set $E$.

It is well-known that every homeomorphism of the Cantor space can be represented by a sequence of digraphs \cite{AkinGlasnerWeiss,DarjiBernardes,shimomuraGraphCover}. Indeed, let $h:\cantor \rightarrow \cantor$ be a homeomorphism of the Cantor space. Let $\{P_n\}$ be a sequence of clopen partitions of $\cantor$ so that $P_{n+1}$ refines $P_n$ and the $mesh (P_n)$ goes to zero as $n \rightarrow \infty$. We define $G_n$ to be a digraph whose vertex set is $P_n$ and whose directed edges are   those $\oa{ab}$, $a, b \in P_n$, for which $h(a) \cap b \neq \emptyset$. Now consider the  inverse limit of digraphs $\{G_n\}$ with bonding maps $\psi_n: P_{n+1} \rightarrow P_{n}$ defined by containment. Then, $\psi_n$ is a surjective (vertex as well as directed edge) graph homomorphism from $G_{n+1}$ onto $G_n$ for which $ \oa {uv} \in G_{n+1}$ implies that $\oa {\psi_n(u),\psi_n(v)} \in G_n$.

Now if we let $\mathcal X = \varprojlim (G_n, \psi_n)$, then $\mathcal X$ is Cantor space topologically. Let $\homeocantor = \{(x, y) \in {X}^2: \oa{x(i)y(i)} \in G_i \}$, we have that $\homeocantor$ is a closed subset of $X^2$. Moreover, $\homeocantor$ is a graph of a function from $X$ to $X$ which is conjugate to $h$.

Motivated by the construction above, we introduce the notion of graph $\cC$-system and topological $\cC$-system.

\begin{definition}\label{def:graphcsystem}
A \emph{graph $\cC$-system} is a  $\csystem$ where $\csystem$ is a $\cC$-structure with the additional properties of $G_n$ being a directed graph and satisfying 
\begin{itemize}
    \item[1.] \label{eq:continuous}  For all $m \in \N$, there is $n > m$ such that for all $g \in G_n$ the following set
        \[ \{\Psi^m_n(g'): \oa{gg'} \in G_n\} \
        \] has cardinality one. 
\end{itemize}
\end{definition}
 A graph $\cC$-system induces a topological $\cC$-system defined as follows.
\begin{definition}
We  say that \emph{$(X,\homeocantor)$} is the \emph{topological $\cC$-system} induced by a graph $\cC$-system $\csystem$ 
if
\[X = \left \{ x \in \Pi_{i=0}^{\infty}G_i: x(n) =\Psi(x(n+1)) \ \forall n  \right \} \textit {and } \homeocantor = \left \{ (x, y) \in X^2:  \oa{x(n)y(n)}  \in G_n \ \forall n \right \}.\]
We use \[ (X,\homeocantor) = \varprojlim\csystem,
\] as a short to denote that $(X,\homeocantor)$  is induced by $\csystem$.
\end{definition}

Note that $X$ is simply the topological inverse limit of the $\csystem$ and as such it is a Cantor space. It inherits subspace topology from the product topology on $\Pi_nG_n$. This topology is generated by the standard metric on $X$ given by $d(x,y)=2^{-n}$ where $n$ is the least integer where $x(n) \neq y(n)$. 

 $\homeocantor$ is a closed subset of $X^2$ whose projection on both coordinates is $X$. Condition 1. of Definition~\ref{def:graphcsystem} implies that  set $\homeocantor$ is the graph of a surjection of $X$.  Moreover, if the following condition is satisfied, then we have that $\homeocantor$ is the graph of a homeomorphism of $X$.

\begin{enumerate}
        \item [{\it 2.}] For all $m \in \N$, there is $n > m$ such that for all $g \in G_n$ the following set
        \[ \{\Psi^m_n(g'): \oa{g'g} \in G_n\} \
        \] has cardinality one. 
\end{enumerate}

As discussed earlier, every continuous surjection of a Cantor space is topologically  conjugate to $(X,\homeocantor)$ generated by some $\cC$-system $\csystem$.

\subsection{${\cF}$-systems}\label{subsec:fsystmes}
Based on graph $\cC$-systems, we introduce $\cF$-systems which capture variety of fences. 

\begin{definition}\label{def:fsystem}
An $\fsystem$ is a \emph{$\cF$-system} if $\fsystem$ is an $\cF$-structure and $\csystem$ is a graph $\cC$-system satisfying the following additional condition.
\begin{enumerate}
    \item for all $ g \in G_n$, there exists  $ g' \in G_{n+1}$ such that $\Psi_n(g') = g$, and $\varphi^L_n(g) = \varphi^L_n(g') \text{ and } \varphi^U_n(g) = \varphi^U_n(g').$
\end{enumerate}
\end{definition}

Note that as $\csystem$ is a part of $\fsystem$, associated with each $\cF$-system  we have a topological $\cC$-system $(X,\homeocantor)$.
Associated with each $\cF$-system \\
$\fsystem$, we have a fence ${\bf F}_{\Phi{}}{}$ determined by $ \Phi = (\varphi ^L, \varphi ^U)$ where \[\varphi^U (x) = \varinjlim \varphi^U_n(x(n)) \text{ and } \varphi^L (x) = \varinjlim \varphi^L_n(x(n)).\]
 Letting $\Phi_n =(\varphi^L_n, \varphi^U_n)$
we obtain that $\cap_{n\in \N} \nfence=\fence$. We use the  maximum metric on $X\times [0,1]$.

Our next aim is to provide conditions on sequences $\{\varphi^L_n\}$ and $\{\varphi^U_n\}$ so as to naturally obtain a continuous surjection of the fence ${\bf F}_{\Phi{}}{}$ which is an extension of the map $(X,\homeocantor)$.

\section{General theorem}\label{sec:General theorem}

In this section, we establish a general realization theorem that provides a systematic method for lifting maps from the Cantor space $\cantor$ to Scissorhand fences $\fence$. 
Under Condition $\Gamma$, our general realization theorem yields a uniquely defined continuous surjection on the resulting fence with the original Cantor map as a factor, while preserving structural properties such as invertibility, Lipschitz regularity, and isometric behavior.

\subsection{Condition $\Gamma$}

Let $\fsystem$ be an $\cF$-system. 
For  $u, v \in G_n$, let 
\begin{equation}
s_n(u, v):= \frac{\varphi^U_n(v)}{\varphi^U_n(u)}
\end{equation}

Let $\oa{uv}\in G_n$. We define 
\begin{equation}\label{def:Fgamma_n}
\begin{split}
\Gamma_n(\oa{uv}):=  \max\Biggl\{ \left | s_n(v,u)  - s_{n+1}(v',u') \right | ,\\ \left | s_n(u,v)  - s_{n+1}(u',v')\right | : 
\oa{u'v'}\in G_{n+1}, \Psi_n(u')=u,  \Psi_n(v')=v\Biggl\}. 
\end{split}
\end{equation}
and let 
\begin{equation}\label{eq:gamma}
\Gamma_n:=\max\{\Gamma_n(\oa{uv}):  \oa{uv}\in G_n\}.
\end{equation}

We will say that $\cF$-system $\fsystem$ satisfies \emph{Condition $\Gamma$} if $\sum^{\infty}_{n=0} \Gamma_n < 1$.

\begin{lemma}\label{lem:bounded}
    If $\fsystem$ satisfies Condition $\Gamma$, then
    \begin{equation}\label{eq:AboundedF} 
    \begin{split}
  0< \inf A \leq \sup A < \infty   \textit{ where } 
 A=\Bigg\{  s_n(u,v), s_n(v,u): \oa{uv}\in G_n,\  n \in \N\Bigg\}.
    \end{split}
\end{equation}
\end{lemma}
\begin{proof}
As $G_0$ has only one element, we have that $s_0(v,u)=1$ for all $u, v \in G_0$. Let $\oa{u'v'} \in G_{n+1}$ and let $\oa{uv} \in G_{n}$ be such that $\Psi_n(u')=u,  \Psi_n(v')=v$. Note that $\left | s_n(v,u)  - s_{n+1}(v',u') \right |$, $\left | s_n(u,v)  - s_{n+1}(u', v') \right | < \Gamma_n$. As   $\sum_{n=0}^{\infty}\Gamma_n< 1$, we have that $\sup A \le 1 + \sum_{n=0}^{\infty} \Gamma_n < \infty$ and $\inf A \ge 1 - \sum_{n=0}^{\infty} \Gamma_n >0 $. 
\end{proof}

Next based on functions $s_n$, we define certain useful functions on $X$,  the inverse limit space $\varprojlim\csystem$. For $n \in \N$ we define $\tilde{s}_n, s:X \rightarrow \mathbb{R}$ by 
\[\tilde{s}_n(x):= s_n(x(n),H_X(x)(n))\]
\begin{equation}\label{eq:s(x)}
s(x):= \lim_{n\to \infty} \tilde{s}_n (x).
\end{equation}
The following simple proposition follows from the definition of Condition $\Gamma$ and verifies that $s$ is well-defined and continuous as each $\tilde{s}_n$ is piecewise constant.
\begin{proposition}\label{prop:sncauchy}
 For all $n \in \N$ and $x \in X$, \[|\tilde{s}_n (x) - \tilde{s}_{n+1} (x)| \le \Gamma_n. \]
\end{proposition}
\begin{lemma}\label{lem:contS}
Suppose that Condition $\Gamma$ is satisfied. Then $s:X \to \R$ is continuous, $s>0$ and if $\varphi^U(x) \neq 0$ then
\[
s(x) = 
 \frac{\varphi^U(H_X(x)) }{\varphi^U(x)}.
\]
\end{lemma}

\begin{proof}
This follows by the definition of $s$ and the fact that $\varphi_n^U(x)$ converges to $\varphi^U(x)$.
\end{proof}

Let \[E^{U} := \{(x, \varphi^U(x)):x \in \cantor\}\] be the set of upper end points.

\begin{theorem}\label{thm:mainF}
Let $\fsystem$ be an $\cF$-system which satisfies Condition $\Gamma$ and  ${\bf F}_{\Phi}$ is a Scissorhand Fence. 
Then, there exists a unique continuous surjection $T: {\bf F}_{\Phi}\rightarrow {\bf F}_{\Phi}$, with $\homeocantor$ as a factor, satisfying \[T\left(x, \varphi^U(x)\right)=\left(\homeocantor(x), \varphi^U(\homeocantor(x))\right)\] for all $x \in X$.
 Moreover, 
 \begin{enumerate}
         \item if $\homeocantor$ is a homeomorphism, then $T$ is a homeomorphism of ${\bf F}_{\Phi}$.
     \item if $s$ as defined in \eqref{eq:s(x)} and $\homeocantor$ are both Lipschitz, then so is $T$.
     \item if $s$ and $1/s$ are both Lipschitz and $\homeocantor$ is bi-Lipschitz, then $T$ is bi-Lipschitz.
     \item if $s=1$ and $\homeocantor$ is an isometry, then $T$ is an isometry.
 \end{enumerate}
\end{theorem}

\begin{proof}
As 
\[T\left(x, \varphi^U(x)\right)=\left(\homeocantor(x), s(x)\varphi^U(x) \right )\] and $s$ is uniformly continuous, we have that $T$ is uniformly continuous on $E^U$. Then, $T$ has a unique continuous extension on the closure of $E^U$, namely ${\bf F}_{\Phi}$. As $E^U$ is dense in ${\bf F}_{\Phi}$ and a subset of the range of $T$, we have that $T$ is a continuous surjection.  

 Now assume that $(X,\homeocantor)$ is a homeomorphism. Note that for any\\
 $\left(x, \varphi^U(x)\right), \left(x', \varphi^U(x')\right)\in {\bf F}_{\Phi}$, if $T\left(x, \varphi^U(x)\right)=T\left(x', \varphi^U(x')\right)$ we get $x=x'$ and consequently $\varphi^U(x)=\varphi^U(x')$. 
Hence $T$ is 1-to-1 on $E^U$ and by the definition of $T$ it is also 1-to-1 on ${\bf F}_{\Phi}$. Similarly, we define
\[\tilde{T}\left(x, \varphi^U(x)\right)=\left(\homeocantor^{-1}(x), \varphi^U(\homeocantor^{-1}(x))\right).\]
As 
\begin{equation}\label{eq:tildeT}
\tilde{T}(x, \varphi^U(x))=\left (\homeocantor^{-1}(x), \frac{1}{s(\homeocantor^{-1}(x))}\varphi^U(x) \right )
\end{equation}
and $1/s$ is uniformly continuous, we have that $\tilde{T}$ is uniformly continuous on $E^U$ and can be extended to a continuous function on ${\bf F}_{\Phi}$. Note that $\tilde{T}T=T\tilde{T}$ is the identity on $E^U$, a dense subset of ${\bf F}_{\Phi}$. Hence, $T$ is the inverse of $\tilde{T}$ and itself is a homeomorphism of $\fence$.

To see (2), let  $M =\alpha + \beta + \gamma $ where $\alpha= \sup_{x \in X} |s(x)|$, $\beta$ is a Lipschitz constant of  $s$ and $\gamma$ is a Lipschitz constant of $\homeocantor$. We will show that $T$ is $M$-Lipschitz. It  suffices to show the $M$-Lipschitz condition on the set $\{(x, \varphi^U(x)):x \in X\}$ as it is dense in ${\bf F}_{\Phi}$. As the metric on ${\bf F}_{\Phi}$ is inherited from the sup metric on $\cantor \times \R$, and $T\left(x, \varphi^U(x)\right)=\left(\homeocantor(x), \varphi^U(\homeocantor(x))\right)$, it suffices to show that both coordinate mappings are Lipschitz. Indeed, the first coordinate function is $\gamma$-Lipschitz. Now we will show that the second coordinate is $(\alpha + \beta)$-Lipschitz.
\begin{align*}
    \left | \varphi^U(\homeocantor(x)) -\varphi^U(\homeocantor(y)) \right | &=  \left | \varphi^U(x) s(x)  -\varphi^U(y) s(y) \right |\\
    & \leq \left | \varphi^U(x) s(x)  -\varphi^U(x) s(y) \right | + \left | \varphi^U(x) s(y)  -\varphi^U(y) s(y) \right | \\
    & \leq  \varphi^U(x) \left |  s(x)  - s(y) \right | + s(y)\left | \varphi^U(x)   -\varphi^U(y)  \right | \\
    & \leq 1 \cdot \beta d(x,y) + \alpha d(\varphi^U(x),\varphi^U(y) )\\
    &  \le ( \beta + \alpha ) d((x,\varphi^U(x)), (y,\varphi^U(y))).
\end{align*}

To see (3) note that $T^{-1}$ is defined on $E^U$ by \eqref{eq:tildeT}. Hence applying part (2) to $\homeocantor^{-1}$ and $\tilde{s}(x)=1/s(\homeocantor^{-1}(x))$, we get the desired result.

To see (4), recall  that $T(x, \varphi^U(x))=(\homeocantor(x),s(x)\varphi^U(x))$. As $s$ is the constant function 1 and $\homeocantor$ an isometry, we have that $T(x, \varphi^U(x))=(\homeocantor(x),\varphi^U(x))$, implying that $T$ is an isometry. 
\end{proof}

\section{Applications to dynamics on isometries}\label{sec: Applications to dynamics on isometries}

In this section we apply the realization theorem, Theorem~\ref{thm:mainF}, to Cantor isometries with nowhere dense orbits. We show that such dynamical systems admit isometric liftings to both the Fraïssé fence and the Lelek fence with some additional control.

We start the section with a simple example which is a special case of more general result on isometries on Fra\"iss\'e fence, Theorem~\ref{thm:liftingisometriesFraisse}.

\begin{example}\label{ex:warmupisometryfraisse}
    There exists a Fra\"iss\'e fence $\fence$, and an isometry $T:\fence\to \fence$ such that the set of periodic points of $T$ is countably infinite and dense in $\fence$ and it is a subset of degenerate components of $\fence$.
\end{example}

\begin{proof}
    We will inductively define an {$\cF$-system} $\fsystem$ satisfying Theorem~\ref{thm:mainF}, where $\{G_n\}$ is a sequence of digraphs whose components are cycles. Moreover, we will ascertain that the sequence $\{\eta_n\}$ as defined in \eqref{eq:eta_n} satisfies $\{\eta_n\} \rightarrow 0$ (in order to apply 5. from  Theorem~\ref{thm:fenceconst}) and $s$ as defined in \eqref{eq:s(x)} satisfies $s=1$ (in order to apply 4. from Theorem~\ref{thm:mainF}).

    At step $k=0$, we let  $G_0$ be a cycle of length $1$. We simply let $\varphi^U_0 =1$ and $\varphi^L_0 =0$ on vertices of $G_0$.   Suppose we are at stage $k$ and $G_k$, $\varphi^L_k$ and $\varphi^L_k$ have been defined so that $\varphi^U_k$ and $\varphi^L_k$ are constant functions on each component of $G_k$ which happens to be cycle. Moreover assume that we have a cycle of length one in $G_k$. We proceed to define $G_{k+1}$, $\Psi_k$, $\varphi^U_{k+1}$ $\varphi^L_{k+1}$. We work with one cycle of $G_k$ at a time. Choose a cycle $C$ of $G_k$. Let $n$ be the length of $C$. Associated with $C$, we define a collection $\mathcal H _C$ of cycles; $\mathcal H_C$ consists of one cycle $D$ of length $n$ and of $(k+2)(k+1)/2$ many cycles of length $2n$ labeled by $C_{i,j}$, $0  \le i < j \le k+1$.
    The map $\Psi_k$ is defined on ${\mathcal H}_C$ in a natural way so that it is a surjective graph homomorphism. 
    On $D$, we simply let $\varphi^U_{k+1}|D$ be the same function as $\varphi^U_k \circ \Psi_k|D$ and $\varphi^L_{k+1}|D=(\varphi^U_k\circ\Psi_k|D + \varphi^L_k\circ \Psi_k|D)/2$. 
    On cycle $C_{i,j}$ we define $\varphi^U_{k+1}$ be the function $l+d\cdot j$ where $l= \varphi_k^L\circ \Psi_k|C_{i, j}$, $d = \frac{1}{k+1} [\varphi_k^U\circ \Psi_k|C_{i,j}-\varphi_k^L\circ \Psi_k|C_{i,j}] $
and $\varphi^L_{k+1}$ be the function $l+d\cdot i$.
We let $G_{k+1}$ be the union of all such $\mathcal{H}_C$'s.

Construction implies that $\{\eta_n\}\to 0$ and thus by Theorem~\ref{thm:fenceconst}(5), we have that $\fence$ is a Fra\"iss\'e fence. By Theorem~\ref{thm:mainF}, we obtain a homeomorphism $T:\fence\to \fence$ with a canonical factor $\homeocantor$. Observe that function $s$ as in Equation~\eqref{eq:s(x)} is identically one and hence by Theorem~\ref{thm:mainF}(4), we have that $T$ is an isometry.

Suppose that $(x,t)\in \fence$ is a periodic point of period $n\in\mathbb N$. Then there exists $k_0$ such that  $x(k)$ is contained in an $n$-cycle of $G_k$ for all $k\geq k_0$. Since the length of \newline $[\varphi^L_{k+1}(x({k+1})), \varphi^U_{k+1}(x({k+1}))]$ is half of the length of $ [\varphi^L_{k}(x(k)), \varphi^U_{k}(x(k))]$, it follows that the component of $x$ is degenerate. Hence, periodic points are contained in degenerate components.

Note that if $C$ is a cycle of $G_k$ of length $n$, then by construction $C$ contains exactly one orbit of $T$ of size $n$. Hence, the set of periodic points is countable. Moreover, every cycle contains a periodic point of $T$. This implies that the set of periodic points is dense in $\fence$.
\end{proof}

The following two propositions will be used later in this section in the proofs of our main theorems about isometries on Lelek and Fra\"iss\'e fence.

\begin{proposition}\label{prop:ClopenIsPeriodic}
     
 Let $h: \cantor \rightarrow \cantor$ be an isometry.
For each clopen set $U \subseteq \cantor$, there exists $l \in \N$ such that $h^l(U)=U$. 
\end{proposition}
\begin{proof} If $U = \cantor$ there is nothing to show. Hence, assume that $U$
 is a proper subset of $\cantor$ and let $\varepsilon>0$ be less than the distance between $U$ and $\cantor \setminus U$.
 Let $B_{r_i}(x_i)$, $1 \le i \le n$ be balls in $U$ which cover $U$ such that $r_i < \varepsilon/4$. Since $C$ is compact and $h$ an isometry every point of $C$ is a recurrent point of $h$. Hence, we have that $(x_1, \ldots, x_n)$ is a recurrent point of $h \times \ldots \times h$ and we may choose positive integer $l$ such $d( h^l(x_i),x_i) < r_i$,  $1 \le i \le n$.
 We claim that  $h^l(U) =U$. Indeed, if $y \in U$, then $y \in B_{r_i}(x_i)$ for some $1 \le i \le n$. We have that  
 
 \[ d(h^l(y), x_i)  \le d(h^l(y), h^l(x_i)) + d( h^l(x_i), x_i) < r_i + r_i < \varepsilon, \]
 verifying that $h^l(U) \subseteq U$. On the other hand, using the fact that $h^l$ and $h^{2l}$ are isometries,  we have that 
 \[ d(h^{-l}(y), x_i)  \le d(h^{-l}(y), h^{-l}(x_i)) + d( h^{-l}(x_i), x_i) = d(y, x_i) + d( h^{l}(x_i), x_i) < r_i + r_i < \varepsilon, \]
 verifying that $h^l(U) \subseteq U$.
 \end{proof}

\begin{proposition}\label{prop:isohelp}
    Let $h: \cantor \rightarrow \cantor$ be an isometry and $\varepsilon >0$. Then, there is a partition ${\mathcal V}$ of $\cantor$, with mesh less than $\varepsilon$ consisting of clopen sets, such that the digraph $({\mathcal V},h)$ consists of cycles. 
\end{proposition}

\begin{proof}
Let $\mathcal U$ be any finite clopen partition of $\cantor$ with mesh less than $\varepsilon$.
For every $U\in\mathcal U$ there is, by Proposition \ref{prop:ClopenIsPeriodic}, some $n\geq 1$ for which $h^n(U)=U$. Since $\mathcal U$ is finite we can assume that the same $n$ works for all $U\in\mathcal U$.
Consider the partition ${\mathcal V}$ of $\cantor$ given by the common refinement of collections $\mathcal U$, $h(\mathcal U),\dots, h^n(\mathcal U)=\mathcal U$, i.e., 
\[ {\mathcal V} = \{ U_0\cap h(U_1)\cap \dots \cap h^{n-1}(U_{n-1}): U_i \in {\mathcal U}
\}.\]
As ${\mathcal U}$ is a partition of  $\cantor$, we have that ${\mathcal V}$ is a partition of $\cantor$. Moreover, if $V \in {\mathcal V}$, then $h(V) \in {\mathcal V}$ as $h^n(U) = U$ for all $U \in {\mathcal U}$. By the definition of ${\mathcal V}$, $h$ is one-to-one on ${\mathcal V}$ and thus a bijection. Hence, $graph ({\mathcal V},h)$ consists of cycles. 

\end{proof}

\begin{theorem}\label{thm:liftingisometriesFraisse}
Let $h \in \groupcantor$ be an isometry such that $orb(x,h)$ is nowhere dense in $\cantor$ for all $x \in \cantor$. Then, there exists $\Phi = (\varphi^L, \varphi^U)$ such that $\fence$ is a Fra\"iss\'e  fence and $\hat{h} \in \groupfence$ is an isometry such that $h$ is a factor of $\hat{h}$. Moreover, if $\{K_n\}$ is a sequence of closed, invariant, nowhere dense subsets of $\cantor$, then we can choose $\Phi = (\varphi^L, \varphi^U)$ so that $\varphi^L = \varphi^U$ on $\bigcup K_n$.
\end{theorem}

\begin{proof}
 We will do this by constructing an $\cF$-system $\fsystem$ which satisfies hypothesis of Theorem~\ref{thm:mainF}(5). Moreover, we will guarantee that the sequence $\{\eta_n\}$, as defined in Equation~\ref{def:etaparameters}, goes to zero, guaranteeing by Theorem~\ref{thm:fenceconst}(5) that $\fence$ is a Fra\"iss\'e fence.

 Let $G_0$ be an open cover consisting  of simply $\cantor$. Let $\varphi_0^L =0$ and $\varphi_0^U=1$. At stage $n$, we will have a clopen partition $G_n$ of $\cantor$ consisting of cycles and $\Phi_n = (\varphi^L_n, \varphi^U_n)$ such that 
 \begin{enumerate}[label=\alph*)]
    \item the mesh of $G_n$ is less than $2^{-n}$,
     \item each of $\varphi_n^U$ and  $\varphi_n^L$ is a constant function on each cycle of digraph $(G_n,h)$, and
     \item  $|\varphi_n^L(x) -\varphi_n^U(x)| < 2^{-n}$  for all $x \in K_n$.
 \end{enumerate}  By the fact that orbit of each $x$ is nowhere dense in $\cantor$ and $K_n$'s are nowhere dense invariant sets, applying  Proposition~\ref{prop:isohelp} to a sufficiently small $\varepsilon >0$, we may choose $G_{n+1}$, a refinement of $G_n$,
 such that the digraph $(G_{n+1},h)$ satisfies
 \begin{itemize}
     \item $G_{n+1}$ consists solely of cycles,
     \item mesh of $G_{n+1}$ is less than $2^{-(n+1)}$,
     \item each cycle of $G_n$ contains at least $2^{2n}$ many cycles from $G_{n+1}$ which are disjoint with $K_{n+1}$.
 \end{itemize} 
 Indeed, the above may be done by choosing a finite set $A \subset \cantor$ such that the orbit closures of any two points in $A$ are disjoint from each other and also from $K_{n+1}$ and each $g \in G_n$ intersects at least  $2^{2n}$ elements of $A$. We simply let $\varepsilon$ be small enough so that the orbit closure of points in $A$ and $K_{n+1}$ are $2\varepsilon$ separated. Now we take $G_{n+1}$ guaranteed by Proposition~\ref{prop:isohelp} with mesh less than $\varepsilon$.
 
 We next define $\varphi_{n+1}^U$ and $\varphi_{n+1}^L$ on $G_{n+1}$. We do this so that for each cycle $H$ in $G_n$ and each $0 \le i < j \le 2^n$, there is a cycle $H'$ of $G_{n+1}$ contained in $H$ such that $\varphi_{n+1}^L = i\cdot 2^{-n} \cdot a +b  $ and $\varphi_{n+1}^U  = j \cdot 2^{-n} \cdot a+b$ where $a$ is the constant value of $\varphi_n^U - \varphi_n^L$ on $H$ and $b = \varphi_n^L $ on $H$. Moreover,
 we guarantee that if a cycle in $G_{n+1}$ intersects $K_{n+1}$ then $\varphi_{n+1}^U - \varphi_{n+1}^L$ on this cycles is less than $2^{-(n+1)}$. All of these can be accommodated by the third condition above. The construction at step $n+1$ is complete. 

 Now we have that our $\cF$-system $\fsystem$ satisfies the following conditions. For all $n \in \N$,
 \begin{itemize}
     \item for all $\oa{uv} \in G_n$, $s_n(u,v) =1$,
     \item $\Gamma_n =0$,
     \item $s:X \rightarrow \R$ is the constant one function.
     \item  $\eta_n$ is less than $2^{-n}$.
 \end{itemize}
By Theorem~\ref{thm:fenceconst}(5), we have that $\fence$ is a Fra\"iss\'e fence. By Theorem~\ref{thm:mainF}, we obtain $\hat{h}$, a homeomorphism of $\fence$, whose canonical factor is $h$. By Theorem~\ref{thm:mainF}(4), we have that $\hat{h}$ is an isometry. Condition c) guarantees that for all $x \in \cup K_n$, we have that $\varphi^L(x) = \varphi^U (x)$. 
\end{proof}

By appropriate modification of the above theorem, we have the following theorem about the Lelek Fence.
\begin{theorem}\label{thm:liftingisometriesLelek}
Let $h \in \groupcantor$ be an isometry such that $orb(x,h)$ is nowhere dense in $\cantor$ for all $x \in \cantor$. Then, there exists $\Phi = (\varphi^L, \varphi^U)$ such that $\fence$ is a Lelek fence and $\hat{h} \in \groupfence$ is an isometry such that $h$ is a factor of $\hat{h}$. Moreover, if $\{K_n\}$ is a sequence of closed, invariant, nowhere dense subsets of $\cantor$, then we can choose $\Phi = (\varphi^L, \varphi^U)$ so that $\varphi^U (x) >0$ for all $ x \in \bigcup K_n$.
\end{theorem}
\begin{proof}
We proceed as in the proof of Theorem~\ref{thm:liftingisometriesFraisse}. As we want to construct a Lelek Fence, we make $\varphi_n^L \equiv 0$, for all $n$. At stage $n$, we have a clopen partition $G_n$ of $\cantor$ and $\Phi_n = (\varphi^L_n, \varphi^U_n)$ satisfying 
\begin{enumerate}[label=\alph*)]
    \item the mesh of $G_n$ is less than $2^{-n}$,
     \item  $\varphi_n^L$ is a constant function on each cycle of digraph $(G_n,h)$, and 
     \item if $g \in G_n$, $h \in G_{n-1}$ with $g \subseteq h$, and $K_n \cap g \neq \emptyset$, then $\varphi^U_n(g) = \varphi^U_{n-1}(h)$.
     \end{enumerate}
     As earlier we construct $G_{n+1}$ so that 
     \begin{itemize}
     \item $G_{n+1}$ consists solely of cycles,
     \item mesh of $G_{n+1}$ is less than $2^{-(n+1)}$,
     \item each cycle of $G_n$ contains at least $2^{n}$ many cycles from $G_{n+1}$ which are disjoint with $K_{n+1}$.
 \end{itemize} 
 We define  $\varphi_{n+1}^U$ so that for each cycle $H$ in $G_n$ and each $0 < i \le 2^n$, there is a cycle $H'$ of $G_{n+1}$ contained in $H$ 
 such that $\varphi_{n+1}^U  = i \cdot 2^{-n} \cdot \varphi_{n+1}^U$ on $H'$. Moreover,
 we guarantee that if a cycle in $G_{n+1}$ intersects $K_{n+1}$ then $\varphi_{n+1}^U = \varphi_{n}^U$. All of these requirements can be accommodated by the third condition above. The construction at step $n+1$ is complete. 

 By Theorem~\ref{thm:fenceconst}(2), we have that $\fence$ is a Lelek fence. By Theorem~\ref{thm:mainF}, we obtain $\hat{h}$, a homeomorphism of $\fence$, whose canonical factor is $h$. By Theorem~\ref{thm:mainF}(5), we have that $\hat{h}$ is an isometry. Condition c) guarantees that for all $x \in \cup K_n$, we have that $\varphi^U (x)>0$.

\end{proof}
\begin{remark}\label{rmk:isoTwosidedScissor} One can modify the proof of Theorem~\ref{thm:liftingisometriesFraisse} slightly so the resulting fence is a two-sided Scissorhand Fence which is not a Fra\"iss\'e Fence. Indeed, when defining $\varphi_{n+1}^U$ and $\varphi_{n+1}^L$ on $G_{n+1}$ as in the proof of Theorem~\ref{thm:liftingisometriesFraisse} if one guarantees that 
the following holds for all $g \in G_{n}$
\begin{enumerate}

    \item there exists $g' \in G_{n+1}$, $g' \subseteq g$ such that $\varphi_{n+1}^L(g') = \varphi_{n}^ L(g)$ , $\varphi_{n+1}^U(g') = \varphi_{n}^U(g)$, 
    \item if $g' \in G$, $g' \subseteq g$ and the above condition does not hold, then $|\varphi_{n+1}^L(g') -\varphi_{n+1}^U(g')| < 1/2 \cdot |\varphi_{n+1}^L(g) -\varphi_{n+1}^U(g)|$,
    \item $\eta^+_{n+1}, \eta^-_{n+1}$ are less than $2^{-(n+1)}$,
\end{enumerate}
 then the resulting fence is a Two-sided Scissorhand Fence which is not a Fra\"iss\'e fence. 
    
\end{remark}

\begin{remark}\label{rem:isometries}
As a corollary of the last two theorems we obtain that there exists an isometry of the Lelek fence (Fra\"iss\'e fence) such that the set of periodic points is dense in the fence and all positive integers are realized as periods. Moreover, one can do this in a fashion so that the periodic points are contained in the $\fence(x)$ where $x$ are periodic points of the Cantor set homeomorphism in the case of Lelek fence and degenerate components in the case of Fra\"iss\'e fence.
\end{remark}

\section{Applications to dynamics on Lelek fence}\label{sec: Applications to dynamics on Lelek fence}

In this section we study the lifting of specific dynamical properties from Cantor systems to the Lelek fence. We show that transitive homeomorphisms of the Cantor space can be lifted to transitive homeomorphisms of the Lelek fence so that a prescribed point on an upper endpoint fiber is transitive, that chaotic Cantor homeomorphisms admit chaotic liftings and that, under an additional recurrence condition, a broad class of topologically mixing Cantor homeomorphisms admits mixing liftings. In particular, this applies to shift homeomorphisms. By collapsing the Cantor base, the corresponding results translate to the Lelek fan.

\begin{theorem}\label{thm:transitive}
    Let $h \in \groupcantor$ be transitive and  $x \in \cantor$ whose orbit is dense in $\cantor$. Then, there exists $\Phi = (\varphi^L, \varphi^U)$ such that $\fence$ is a Lelek fence and $\hat{h} \in \groupfence$ such that 
    \begin{enumerate}
        \item[a)] $h$ is the canonical factor of $\hat{h}$, and
        \item [b)] $(x, \varphi^U(x))$ is a transitive point of $\hat{h}$.
    \end{enumerate}
\end{theorem}
\begin{proof}
We let $\varphi^L=0$ and 
    we will define a sequence of partitions $\{G_n\}$  of $\cantor$ and a sequence of functions $\{\varphi^U_n\}$, so that $\fsystem$ will be an  $\cF$-system associated with $h$. In particular, $
    \Psi_n:G_{n+1} \rightarrow G_n$ is the  containment map and $ \oa{uv} \in G_n$ if and only if $h(u) \cap v \neq \emptyset$.  This system will satisfy Condition $\Gamma$ and $\fence$ will be a Lelek fence. Applying Theorem~\ref{thm:mainF}, we will obtain $\hat h$ that is a homeomorphism of $\fence$. Then, we will verify that $\hat{h}$ satisfies conclusion b).

 Let $x_n = h^n(x)$ for all $n \in \Z$. The general strategy is as follows: For each $n$, we will choose $\ell_n \in \N$ and an open set $U\subset \cantor$ containing $x_{0}$ such that $h^i(U) \cap h^{i'}(U) = \emptyset$ for $ -\ell_n  \le i< i' \le \ell_n$. We will extend the collection $\{h^i(U): -\ell_n \le i \le \ell_n\}$ to a clopen partition $G_n$ of $ \cantor$. Function $\varphi^U_n$ will be defined on $G_n$ in an appropriate fashion so that the resulting $\Phi$ has the property that $\fence$ is a Lelek fence. All of this will be done by induction, taking into account the previous stage of the construction.  

    Let $G_1= \{\cantor\}$ and $\varphi^U_1:G_1 \rightarrow (0,1]$ be defined as $\varphi^U_1(\cantor) =1$.

    Suppose that $G_n$, $\varphi^U_n: G_n \rightarrow (0,1]$ and auxiliary parameters $\ell_n \in \N$ have been defined. We proceed to define $G_{n+1}$, $\varphi^U_{n+1}: G_{n+1} \rightarrow (0,1]$ and $\ell_{n+1}$. 
    
  Let \[  N =  \max  \left \{ \frac{\varphi^U_n(g_1)}{\varphi^U_n(g_2)}: \oa{g_1g_2} \textit { or } \oa{g_2g_1} \in G_n \right \}.\]
    
    Choose $0 < r <1$ so that each of $(1-r)N, |1-r^{-1}|N$ is less than $2^{-(n+1)}$. Now choose $ t \in \N$ sufficiently large so that $r^{t-1}< 2^{-(n+1)}$. Observe that  $\{\varphi^U_n(g), r\varphi^U_n(g), r^2 \varphi^U_n(g), \ldots,\\ r^{t-1} \varphi^U_n(g)\}$ is $2^{-(n+1)}$ dense in $[0, \varphi^U_n(g)]$, for all $g \in G_n$.

Observe that for all $\oa{g_1g_2} \in G_n$, there are infinitely many $j$'s  such that $x_j \in g_1$ and $x_{j+1} \in g_2$. This indeed holds since $\{x_j\}_{j \ge 0}$ is dense in $\cantor$. Using this fact, we may choose $\ell_{n+1} > \ell_n$ and an increasing sequence of integers $\{m_i\}_{i=-t}^t$ with $m_0 =0$, $m_{-1} < -\ell_n$ $m_1> \ell_n$,  $m_{-t} = -\ell_{n+1}$, $m_t = \ell_{n+1}$,  so that for all    $\oa{g_1g_2} \in G_n$ and for all $-t \le i < t$  there exists $j$ with $m_i \le j < m_{i+1}$ such that $x_j \in g_1$ and $x_{j+1} \in g_2$. In other words, intuitively speaking, for all $-t \le i < t$, $\{x_{m_i}, \ldots x_{m_{i+1}-1}\}$ realizes all the edges of $G_n$.
     
 Now choose  a clopen set $U$, containing $x_{0}$ such that the collection $\{h^j(U):-\ell_{n+1} \le j \le \ell_{n+1}\}$ is pairwise disjoint and refines $G_n$. We may do this as $\{x_n\} =\{h^n(x)\}$ and $x$ is not periodic. Let $H_i =\{h^j(U): m_i \le j <m_{i+1} \}$ for $-t \leq i<t$. Finally, enlarge the collection $\cup_{i=-t}^{t-1}H_i$ to form a clopen partition $G_{n+1}$ which refines $G_n$ and such that $mesh (G_{n+1})< 2^{-(n+1)}$. 
    We now define $\varphi^U_{n+1}: G_{n+1} \rightarrow (0,1]$. Let $g \in G_{n+1}$ with $g \subseteq g'$, $g' \in G_n$. Then,
    \[ \varphi^U_{n+1}(g) = \begin{cases} 
          \varphi^U_n(g') r^{i }&  \textit { if }  \ \ g \in H_{i},\ \ \ 
           0 \le i < t  \ \ \ \\
           \varphi^U_n(g') r^{|i| -1 }&  \textit { if }  \ \ g \in H_{i},  \ \ \ 
           -t \le i < 0 \\
          \varphi^U_n(g') r^{t-1} & \textit{ otherwise }
       \end{cases}
    \]

    We claim that the following conditions hold at stage $n+1$. 
    \begin{enumerate}
        \item for all $ g\in G_{n+1}$, if $g \cap \{x_{-\ell_{n}}, \ldots, x_0,\ldots, x_{\ell_{n}}\} \neq \emptyset$, then $\varphi^U_{n+1}(g) = \varphi^U_{n}(g') $
        where $g' \in G_n$ is such that $g \subseteq g'$.
        \item if $\oa{g_1g_2}$ or  $\oa{g_2g_1}\in G_{n+1}$, then 
        \[  \left | \frac{\varphi^U_{n+1}(g_1)}{\varphi^U_{n+1}(g_2)}  -\frac{\varphi^U_n(g_1')}{\varphi^U_n(g_2')}\right | < 2^{-(n+1)}
        \]
        where $g_i \subseteq g_i'$, $g_i' \in G_n$ for $i=1,2$.

        \item For all $g \in G_n$ and $ 0\le s < t$, there is $g' \in G_{n+1}$ such that $\varphi^U_{n+1} (g')=r^s\varphi^U_n(g)$.
  
    \end{enumerate}

Condition 1. holds from the definition of $\varphi^U_{n+1}$ and the fact that $\ell_{n}< m_1$ and $m_{-1} < -\ell_n$.

Let us now verify Condition 2. Suppose $\oa{g_1g_2} \in G_{n+1}$. Let $g_i' \in G_n$, $i =1 ,2$, be such that $g_i \subseteq g_i'$. If neither of $g_1$ or $g_2$ belongs to $\cup_{i=-t}^{t-1} H_i$, then by definition of $\varphi^U_n$, we have that $\varphi^U_{n+1}(g_i) = r^{t-1} \varphi^U_n(g_i')$ and the Condition 2. is verified in this case.
If both of $g_1, g_2 $ belong to $\cup_{i=-t}^{t-1} H_i$, by the manner in which $U$ was chosen we have that $g_1 \in H_{a}$,  $g_2 \in H_{b}$ for some $a, b$ with $|a-b|\le 1$.
Now the worst case scenario is where $\varphi^U_{n+1}(g_1)=r^a \varphi^U_{n}(g'_1) $ and $\varphi^U_{n+1}(g_2)=r^b \varphi^U_{n}(g'_2) $ where $|a-b|=1$.  Then,
\[  \left | \frac{\varphi^U_{n+1}(g_1)}{\varphi^U_{n+1}(g_2)}  -\frac{\varphi^U_n(g_1')}{\varphi^U_n(g_2')}\right | \le \max  \left \{ (1-r), |1-r^{-1}| \right \} \cdot \frac{\varphi^U_n(g_1')}{\varphi^U_n(g_2')} < 2^{-(n+1)}.
        \]

Now suppose that $g_1 \notin \cup_{i=-t}^{t-1} H_i$ but $g_2 \in \cup_{i=-t}^{t-1} H_i$. In this case, we have $g_2 \in H_{-\ell_{n+1}}$ and we have that $\varphi^U_{n+1}(g_i) = r^{t-1} \varphi^U_n(g_i')$, $i =1,2$ and we are done. The case where $g_2 \notin \cup_{i=-t}^{t-1} H_i$ but $g_1 \in \cup_{i=-t}^{t-1} H_i$ is symmetric. In this case we have that $g_1 \in H_{\ell_{n+1}} $ and the argument precedes as above.

Condition 3 follows from the fact that  for all $-t \le i < t$ for all $ g \in G_n$, there exists $g' \in H_i$ such that $g' \subseteq g$.

That $\fsystem$ is an $\cF$-system follows from the definition of $\varphi^U_{n+1}$. Finally, define $\varphi^U = \lim \varphi^U_n$. Condition 2. above implies that the $\Gamma_n < 2^{-(n+1)}$ and hence the condition $\Gamma$ holds. Furthermore, since  $\{\varphi^U_n(g), r\varphi^U_n(g), r^2 \varphi^U_n(g), \ldots, r^{t-1} \varphi^U_n(g)\}$ is $2^{-(n+1)}$ dense in $[0, \varphi^U_n(g)]$, for all $g \in G_n$, Condition 3. implies that $\eta^+_n < 2^{-(n+1)}$. Hence by Theorem~\ref{thm:fenceconst} we have that $\fence$ is a Lelek Fence. Thus we have that $\fsystem$ is an $\cF$-system satisfying Condition $\Gamma$ with $\{(x, \varphi^U(x)\}$ is dense in $\fence$. Now applying Theorem~\ref{thm:mainF}, we obtain that the resulting $\hat{h}$ is a homeomorphism of $\fence$ whose canonical factor is $h$.  Finally, we verify Conclusion b). We note that, by construction, $  \{ (x_i, \varphi^U(x_i)): - \ell_{n+1} \le  i \le \ell_{n+1}\}$ is $2\cdot 2^{-(n+1)}$ dense in $\nfence$. Moreover, $\varphi^U \le \varphi^U_n$, and  Condition 1 implies that $\varphi^U_{n+1}(x_i) = \varphi^U(x_i)$, $- \ell_{n+1} \le  i \le \ell_{n+1}$. Putting these facts together, we have that $\{(x_i, \varphi^U (x_i) \}$ is dense in $\fence$. As $\hat{h}$ is a homeomorphism and  $x_i =h ^i(x)$, we have that the orbit  of $(x, \varphi^U (x))$ under $\hat {h}$ is $\{(x_i, \varphi^U (x_i) \}$, completing the proof.

\end{proof}
 \begin{theorem}\label{thm:chaotic}
        Let $h \in \groupcantor$ be chaotic. Then, there exist $\Phi = (\varphi^L, \varphi^U)$ such that $\fence$ is a Lelek fence and a chaotic $\hat{h} \in \groupfence$ so that $h$ is the canonical factor of $\hat{h}$.
    \end{theorem}
    \begin{proof}
    Recall that in the setting of compact metric space $X$ a map is chaotic if and only if for every $\eps>0$ there exists a periodic point whose orbit is $\eps$-dense in $X$. We will use this definition throughout the proof.

        We will appropriately modify the proof of Theorem~\ref{thm:transitive} to obtain the desired homeomorphism.
        As in the Theorem~\ref{thm:transitive}, we let $\varphi^L=0$ and we define a sequence of partitions $\{G_n\}\subset \cantor$ with mesh less than $2^{-n}$ and $\Phi$, so that $\fsystem$ is an  $\cF$-system associated with $h$ satisfying Conditions $\Gamma$ and that $\fence$ is a Lelek fence. 

         The general strategy is as follows: For each $n$, we will choose a periodic point $p_n$ of $h$, $\ell_n \in \N$ and an open set $U\subset \cantor$ containing $p_n$ such that $h^i(U) \cap h^{i'}(U) = \emptyset$ for $ -\ell_n  \le i< i' \le \ell_n$. We will extend the collection $\{h^i(U): -\ell_n \le i \le \ell_n\}$ to a clopen partition $G_n$ of $ \cantor$. Function $\varphi^U_n$ will be defined on $G_n$ in an appropriate fashion so that the resulting $\Phi$ has the property that $\fence$ is a Lelek fence. We will again use inductive approach.
         
 Let $G_1= \{\cantor\}$ and $\varphi^U_1:G_1 \rightarrow (0,1]$ be defined as $\varphi_1(\cantor) =1$.
 
          Suppose that $G_n$, $\varphi^U_n: G_n \rightarrow (0,1]$ and auxiliary parameters $\ell_n \in \N$ and periodic points $p_1,\ldots, p_n$ have been defined. We proceed to choose periodic point $p_{n+1}$ and define $G_{n+1}$. Using the fact that $h$ is chaotic, we may choose a periodic point $p_{n+1}$ not intersecting orbits of $\{p_1, \ldots p_n\}$ and $ \ell_{n+1} > \ell_n$ and an increasing sequence of integers $\{m_i\}_{i=-t}^t$ with $m_0 =0$,   $m_{-t} = -\ell_{n+1}$, $m_t = \ell_{n+1}$,  so that for all    $\oa{g_1g_2} \in G_n$ and for all $-t \le i < t$  there exists $j$ with $m_i \le j < m_{i+1}$ such that $h^j(p_{n+1}) \in g_1$ and $h^{j+1}(p_{n+1}) \in g_2$. In other words, intuitively speaking, for all $-t\leq i< t$, $\{h^{m_i}(p_{n+1}), \ldots, h^{m_{i+1}-1}(p_{n+1})\}$ realizes all the edges of $G_n$. 

    Now choose  a clopen set $U$, containing $h^{-\ell_{n+1}}(p_{n+1})$ such that the collection $\{h^j(U):-\ell_{n+1} \le j \le \ell_{n+1}\}$ is pairwise disjoint, refines $G_n$, contains no points of the orbits of $\{p_1, \ldots, p_n\}$ and has mesh less than $2^{-(n+1)}$.   Let $H_i =\{h^j(U): m_i \le j <m_{i+1} \}$ for $-t \leq i<t$. Finally, enlarge the collection $\cup_{i=-t}^{t-1}H_i$ to form a clopen partition $G_{n+1}$, with mesh less than $2^{-(n+1)}$ which refines $G_n$. 
    We now define $\varphi^U_{n+1}: G_{n+1} \rightarrow (0,1]$. Let $g \in G_{n+1}$ with $g \subseteq g'$, $g' \in G_n$. Then,
    \[ \varphi^U_{n+1}(g) = \begin{cases} 
          \varphi^U_n(g') r^{t-1-i }&  \textit { if }  \ \ g \in H_{i},\ \ \ 
           0 \le i < t  \ \ \ \\
           \varphi^U_n(g') r^{t-|i| }&  \textit { if }  \ \ g \in H_{i},  \ \ \ 
           -t \le i < 0 \\
          \varphi^U_n(g')  & \textit{ otherwise }
       \end{cases}
    \]

      We claim that the following conditions hold at stage $n+1$. 
    \begin{enumerate}
        \item for all $ g\in G_{n+1}$, if $g\cap \cup^{n}_{i=1}\orb(p_i,h)\neq \emptyset$ then $\varphi^U_{n+1}(g) = \varphi^U_{n}(g') $
        where $g' \in G_n$ is such that $g \subseteq g'$.
        \item if $\oa{g_1g_2}$ or  $\oa{g_2g_1}\in G_{n+1}$, then 
        \[  \left | \frac{\varphi^U_{n+1}(g_1)}{\varphi^U_{n+1}(g_2)}  -\frac{\varphi^U_n(g_1')}{\varphi^U_n(g_2')}\right | < 2^{-(n+1)}
        \]
        where $g_i \subseteq g_i'$, $g_i' \in G_n$ for $i=1,2$,

        \item For all $g \in G_n$ and $ 0\le s < t$, there is $g' \in G_{n+1}$ such that $\varphi^U_{n+1} (g')=r^s\varphi^U_n(g)$.
  
    \end{enumerate}

These three conditions are verified analogously as in Theorem~\ref{thm:transitive}.    

That $\fsystem$ is an $\cF$-system follows from the definition of $\varphi^U_{n+1}$. 
Finally, define $\varphi^U = \lim \varphi^U_n$. Condition 2. above implies that the $\Gamma_n < 2^{-(n+1)}$ and hence  Condition $\Gamma$ holds. Furthermore, similarly as in Theorem~\ref{thm:transitive}, Condition 3. implies that $\eta^+_n < 2^{-(n+1)}$. Therefore, by Theorem~\ref{thm:fenceconst} we have that $\fence$ is a Lelek Fence. Thus we have that  $\fsystem$ is an $\cF$-system satisfying Condition $\Gamma$ with $\{(x, \varphi^U(x)\}$ is dense in $\fence$.
Now applying Theorem~\ref{thm:mainF}, we obtain that the resulting $\hat{h}$ is a homeomorphism of $\fence$ whose canonical factor is $h$. Now we observe that $\hat h$ is chaotic. It suffices to show that periodic orbits are arbitrarily dense in $\fence$. To this end, note that \[
\{(h^{i}(p_{n+1}), \varphi^U_{n+1}((h^{i}(p_{n+1})): -\ell_{n+1} \le i \le \ell_{n+1} \}\] is $2^{-n}+ 2^{-(n+1)}$ dense in the $n+1$ stage of construction of $\fence$. This follows from our construction of $\varphi^U_{n+1}$ and the fact that the mesh of $G_n$ is less than $2^{-n}$.
Now by Condition 1., it follows that the orbit of $(p_{n+1}, \varphi^U(p_{n+1}))$ is $2^{-n}+ 2^{-(n+1)}$ 
dense in $\fence$, yielding that $\hat{h}$ is chaotic.
    \end{proof}

\begin{remark}
Theorem~\ref{thm:transitive} and Theorem~\ref{thm:chaotic} take different approaches. In Theorem~\ref{thm:transitive}, at each step we consider a part of a dense orbit and consider its extension in successive steps. On the other hand, in Theorem~\ref{thm:chaotic} we take a new finite orbit at each step and use it to define $H_i$'s. As at stage $n+1$, we want to preserve the behavior of finite orbits considered at the previous stages, in Theorem~\ref{thm:chaotic}, $\varphi^U_{n+1}$ is  defined in a manner that is in some sense reverse  from Theorem~\ref{thm:transitive}.
\end{remark}
\begin{remark}
From Theorem~\ref{thm:chaotic} one can show that there are uncountably many pairwise non-conjugate chaotic maps of Lelek Fence as well as of Lelek fan. Indeed, let $A \subset \N$ be an infinite set consisting of primes. Let $X_A := \Pi _{n \in A} \Z_n$ and let $h_n:\Z_n \rightarrow \Z_n$ be the $+1$ map modulo $n$. Then, defining $h_A: X_A \rightarrow X_A$ as the product map $\Pi_{n \in A} h_n$, we have that $(X_A,h_A)$ is a chaotic Cantor system such that the set of periods of periodic points of $h_A$ is $A$. Now, by Theorem~\ref{thm:chaotic}, there is $\widehat{h_A}$ on some Lelek fence $\fence(A)$ whose set of periods of periodic points is $A$. Moreover, we have that for each $n \in A$, there is periodic point of $\widehat{h_A}$ which is not in the base of $\fence(A)$. Hence, identifying the base, we have that there is a Lelek fan and a homeomorphism of it whose set of periodic points is $A$. As the set of periodic points is preserved under conjugation and there are uncountably many distinct such sets $A$, we obtain the desired result. 
\end{remark}

\begin{theorem}\label{thm:mixing}
Let $h:\cantor \rightarrow \cantor$ be a mixing homeomorphism satisfying the following conditions:
for every pair of non-empty open sets \( U, V \subseteq \cantor \), there exists a  compact set \( K_{U,V}\subset \cantor  \) and $m(U,V) \in \N$ such that:
\begin{itemize}
\item $K_{U,V}$ is nowhere dense in $\cantor$ and $h(K_{U,V} )=K_{U,V}$,
    \item for all \( m \geq m(U,V) \), \( h^m(K_{U,V} \cap U) \cap V \neq \emptyset \).
\end{itemize}

Then $h$ admits a lifting $\hat{h}$ to the Lelek fence which is a mixing homeomorphism.
\end{theorem}
\begin{proof}
We let $\varphi^L=0$ and 
    we will define a sequence of partitions $\{G_n\}$  of $\cantor$ and a sequence of functions $\{\varphi^U_n\}$, so that $\fsystem$ will be an  $\cF$-system associated with $h$. In particular, $
    \Psi_n:G_{n+1} \rightarrow G_n$ is the  containment map and $ \oa{uv} \in G_n$ if and only if $h(u) \cap v \neq \emptyset$.  This system will satisfy Condition $\Gamma$ and $\fence$ will be a Lelek fence. Applying Theorem~\ref{thm:mainF}, we will obtain $\hat h$ that is a homeomorphism of $\fence$. Then, we will verify that $\hat{h}$ is mixing.

     Let $G_1= \{\cantor\}$ and $\varphi^U_1:G_1 \rightarrow (0,1]$ be defined as $\varphi^U_1(\cantor) =1$. By hypothesis, let $K_1=K_{\cantor,\cantor}$ and let $m_1 = m(\cantor,\cantor)$.
     
Suppose we are at stage $n$,  compact set $K_n$, $\varphi^U_{n}$, integer $m_n$ and  clopen partitions $G_n$  have been defined so that the following conditions hold.
        \begin{enumerate}[label=(\alph*)]
            \item $h(K_n) = K_n$ and $K_n$ is nowhere dense in $\cantor$.
    \item For all $g, g' \in G_n$ and all $m \ge m_n$ we have that $h^m(K_n \cap g ) \cap g' \neq \emptyset$.
    \item  For all $ g\in G_n$  with $g \cap K_{n-1}  \neq \emptyset$ and $g' \in G_{n-1}$ with $g \subseteq g'$ we have that  $ \varphi^U_{n}(g)= \varphi^U_{n-1}(g')$.
        \end{enumerate}

        We now describe how to construct $K_{n+1}$, $\varphi^U_{n+1}$, $m_{n+1}$ and $G_{n+1}$.

  Let \[  N =  \max  \left \{ \frac{\varphi^U_n(g_1)}{\varphi^U_n(g_2)}: \oa{g_1g_2} \textit { or } \oa{g_2g_1} \in G_n \right \}.\]
    
    Choose $0 < r <1$ so that each of $(1-r)N, |1-r^{-1}|N$ is less than $2^{-(n+1)}$. Now choose $ t \in \N$ sufficiently large so that $r^{t-1}< 2^{-(n+1)}$. Observe that  $\{\varphi^U_n(g), r\varphi^U_n(g), r^2 \varphi^U_n(g), \ldots,\\
    r^{t-1} \varphi^U_n(g)\}$ is $2^{-(n+1)}$ dense in $[0, \varphi^U_n(g)]$, for all $g \in G_n$.

As $h$ is transitive, there is $x \in \cantor$ such that $\orb(h,x)$ is dense in $\cantor$ and  $\{h^j(x): j \in \Z \} \cap K_n = \emptyset$. Let $x_j = h^j(x)$, $j \in \Z$.
Observe that for all $\oa{g_1g_2} \in G_n$, there are infinitely $j$'s  such that $x_j \in g_1$ and $x_{j+1} \in g_2$. This indeed holds since $\{x_j\}_{j \ge 0}$ is dense in $\cantor$. Using this fact, we may choose $\ell\in \N$ and an increasing sequence of integers $\{m_i\}_{i=-t}^t$ with $m_0 =0$,   $m_{-t} = -\ell$, $m_t = \ell$,  so that for all    $\oa{g_1g_2} \in G_n$ and for all $-t \le i < t$  there exists $j$ with $m_i \le j < m_{i+1}$ such that $x_j \in g_1$ and $x_{j+1} \in g_2$. In other words, intuitively speaking, for all $-t \le i < t$, $\{x_{m_i}, \ldots x_{m_{i+1}-1}\}$ realizes all the edges of $G_n$.
     
 Now choose  a clopen set $U$, containing $x_0$ such that the collection $\{h^j(U):-\ell \le j \le \ell \}$ is pairwise disjoint, refines $G_n$ and is disjoint from $K_n$. We may do this as $\{x_n\} \cap K_n = \emptyset$. Let $H_i =\{h^j(U): m_i \le j <m_{i+1} \}$ for $-t \leq i<t$. Finally, enlarge the collection $\cup_{i=-t}^{t-1}H_i$ to form a clopen partition $G_{n+1}$ which refines $G_n$ and such that $mesh (G_{n+1})< 2^{-(n+1)}$. The definition of $G_{n+1}$ is complete.
 
    We now define $\varphi^U_{n+1}: G_{n+1} \rightarrow (0,1]$. Let $g \in G_{n+1}$ with $g \subseteq g'$, $g' \in G_n$. Then,
    \[ \varphi^U_{n+1}(g) = \begin{cases} 
          \varphi^U_n(g') r^{t-1-i }&  \textit { if }  \ \ g \in H_{i},\ \ \ 
           0 \le i < t  \ \ \ \\
           \varphi^U_n(g') r^{t-|i| }&  \textit { if }  \ \ g \in H_{i},  \ \ \ 
           -t \le i < 0 \\
          \varphi^U_n(g')  & \textit{ otherwise }
       \end{cases}
    \]
    In particular, if $g \cap K_n  \neq \emptyset$, then $ \varphi^U_{n+1}(g)= \varphi^U_n(g')$.

  Let  $K_{n+1} = \cup_{g_1,g_2 \in  G_{n+1}} K_{g_1,g_2}$ and $m_{n+1} > \max\{m_n, m(g,g'): g,g' \in G_n \}$. Then, $K_{n+1}$ and $m_{n+1}$ have properties (a)-(c) hold at stage $n+1$ and the induction step $n+1$ is complete. 

We claim that the following conditions hold at stage $n+1$.
    \begin{enumerate}
        \item for all $ g\in G_{n+1}$, if $g\cap K_n \neq \emptyset$ then $\varphi^U_{n+1}(g) = \varphi^U_{n}(g') $
        where $g' \in G_n$ is such that $g \subseteq g'$.
        \item if $\oa{g_1g_2}$ or  $\oa{g_2g_1}\in G_{n+1}$, then 
        \[  \left | \frac{\varphi^U_{n+1}(g_1)}{\varphi^U_{n+1}(g_2)}  -\frac{\varphi^U_n(g_1')}{\varphi^U_n(g_2')}\right | < 2^{-(n+1)}
        \]
        where $g_i \subseteq g_i'$, $g_i' \in G_n$ for $i=1,2$,
 
        \item For all $g \in G_n$ and $ 0\le s < t$, there is $g' \in G_{n+1}$ such that $\varphi^U_{n+1} (g')=r^s\varphi^U_n(g)$.
        
    \end{enumerate}

These three conditions are verified analogously as in Theorem~\ref{thm:transitive}.    

That $\fsystem$ is an $\cF$-system follows from the definition of $\varphi^U_{n+1}$. 
Finally, define $\varphi^U = \lim \varphi^U_n$. Condition 2. above implies that the $\Gamma_n < 2^{-(n+1)}$ and hence  Condition $\Gamma$ holds. Furthermore, similarly as in Theorem~\ref{thm:transitive}, Condition 3. implies that $\eta^+_n < 2^{-(n+1)}$. Therefore, by Theorem~\ref{thm:fenceconst} we have that $\fence$ is a Lelek Fence. Thus we have that  $\fsystem$ is an $\cF$-system satisfying Condition $\Gamma$ with $\{(x, \varphi^U(x)\}$ is dense in $\fence$.
Now applying Theorem~\ref{thm:mainF}, we obtain that the resulting $\hat{h}$ is a homeomorphism of $\fence$ whose canonical factor is $h$. 

Now we observe that $\hat h$ is mixing. First, by construction, we have that $\varphi ^U(y) = \varphi^U_n (y)$ for all $y \in K_n$. Recall that the topology on $\fence$ is inherited from the product topology on $\cantor \times [0,1]$. Let $u_i \times (a_i,b_i)$ be open in  $\cantor \times [0,1]$ such that $ u_i \times (a_i,b_i) \cap \fence  \neq \emptyset$, $i=1,2$. It will suffice to show that there exists $n \in \N$, for  all  $m \ge m_n$ we have that 
\[
\hat{h}^{m} (u_1 \times (a_1,b_1) \cap \fence) \cap (u_2 \times (a_2,b_2) \cap \fence)  \neq \emptyset.
\]
To this end, let  $y_i \in \cantor$ such that $(y_i, \varphi^U(y_i)) \in \fence\cap (u_i \times (a_i,b_i))$, $i=1,2$. As $\varphi^U$ is upper semicontinuous, we have that $\varphi^U ((-\infty, b_i))$ is an open subset of $\cantor$ containing $y_i$, $i=1,2$. As $\varphi^U$ is the limit of $\{\varphi^U_n\}$, there exists $n\in \N$ such that  we have $g_i \in G_n$ with $y_i \in g_i \subseteq u_i$ has the property that $g_i \times \{\varphi^U_n(g_i)\} \subseteq u_i \times (a_i,b_i)$ for $i=1,2$. By condition (b), for all  $m \ge m_n$, there is $y\in K_n \cap g_1$ such that  $h^m(y) \in g_2$. Note that 
\[(y, \varphi^U_n(y)) \in u_1 \times (a_1,b_1) 
\]
and 
\[(h^m(y), \varphi^U_n(h^m(y))) \in u_2 \times (a_2,b_2).
\]
As $y \in K_n$, we have that $\varphi^U(y) =\varphi^U_n(y)$ and $h(K_n) =K_n$, we have that $\varphi^U(h^m(y)) =\varphi^U_n(h^m(y))$. Hence, we have 
\[ (y, \varphi^U(y)) = (y, \varphi^U_n(y)) \in u_1 \times (a_1,b_1) \]
and \[ (h^m(y), \varphi^U(h^m(y))) = (h^m(y), \varphi^U_n(h^m(y))) \in u_2 \times (a_2,b_2). \]  Note that \[(y, \varphi^U(y)) \in (u_1 \times (a_1,b_1) \cap \fence) 
\]
and  
\[\hat{h}^m ((y, \varphi^U(y))) = (h^m(y), \varphi^U(h^m(y)))  \in (u_2 \times (a_2,b_2) \cap \fence),
\]
verifying that $\hat{h}$ is mixing and completing the proof.
\end{proof}

\begin{example}\label{rmk:mixing}
    Let us consider $h$ the shift map on the Cantor set $\cantor$ consisting of $\{0,\ldots, n-1\}^{\Z}$, where  $n \ge 2$. Then, $h$ is mixing, and we will show in the following that $h$ also satisfies the hypothesis of Theorem~\ref{thm:mixing}. Without loss of generality, we may assume that $U=[u_{-n}\ldots u_n]$ and $V=[v_{-n}\ldots v_n]$. For each $k \in \N$, we define $w_k \in \cantor$ as follows: 

   \begin{align*}
        w_k|_{[-n,n]} &=  u_{-n}\ldots u_n\\
         w_k|_{[n+k, 3n+k]} &= v_{-n}\ldots v_n\\
         w_k (i) & = 0  \ \ \textit {otherwise}\end{align*} 
         Moreover, let  
         \begin{align*}
        u_{\infty}|_{[-n,n]} &=  u_{-n}\ldots u_n\\
         u_{\infty} (i) &= 0  \ \
         \textit{otherwise}
         \\
         v_{\infty}|_{[-n,n]} &=  v_{-n}\ldots v_n\\
         v_{\infty} (i) &= 0  \ \ \textit{otherwise}
         \end{align*} 
         Then, $L = \{w_k: k \in \N\} \cup \{ u_{\infty}\}$ is a compact subset of $U$ such that for all $m \ge 2n+1$, we have that $h^m(L) \cap V \neq \emptyset$. Moreover,  $K_{U,V}:=\overline{\bigcup_{n \in \mathbb{Z}} h^n(L)}$ is  a countable  invariant set consisting of $\{ h^l(w_k), h^l(u_{\infty}), h^l(v_{\infty}): l \in \Z \} \cup \{\underline{0}\}$, and hence nowhere dense.
\end{example}

\section{Applications to dynamics on Fra\"iss\'e fence}\label{sec: Applications to dynamics on Fraisse fence}

In this section, we first generalize the realization theorem, Theorem~\ref{thm:mainF}, to obtain finer control over the resulting dynamics, readily applicable to two-sided Scissorhand fences (Theorem~\ref{thm:main}). As an application, we show that odometer Cantor systems admit liftings to minimal homeomorphisms of the Fra\"iss\'e fence. Moreover, this construction yields uncountably many pairwise non-conjugate minimal homeomorphisms, none of which factor onto another.

Let $\fsystem$ be an $\cF$-system.
For  $u,v \in G_n$, let $s_n (u,v)$ be the affine function which maps interval $[\varphi_n^L(u),\varphi_n^U(u)]$ onto $[\varphi_n^L(v),\varphi_n^U(v)]$. More explicitly, 
\begin{equation}
s_n (u,v)(t):=    \frac{\varphi_n^U(v)-\varphi_n^L(v)}{\varphi_n^U(u)-\varphi_n^L(u)}  (t - \varphi_n^L(u) ) + \varphi_n^L(v)
\end{equation}

Let $\oa{uv}\in G_n$. We define 
\begin{equation}\label{def:Fgamma_n_new}
\begin{split}
\Gamma_n(\oa{uv}):=  \max\Biggl\{ \left | s_n (u,v) (t)   - s_{n+1} (u',v') (t) \right | : \\
\oa{u'v'}\in G_{n+1}, \Psi_n(u')=u,  \Psi_n(v')=v, t \in [\varphi_{n+1}^L(u'),\varphi_{n+1}^U(u')]\Biggl\}. 
\end{split}
\end{equation}

As $s_n(u,v)$ and $s_{n+1} (u',v')$ are affine functions, we have that $| s_n (u,v) (t)   - s_{n+1} (u',v') (t)|$ is maximal at one of the  endpoints of the domain of $s_{n+1} (u',v')$, namely $t  \in \{ \varphi_{n+1}^L(u'),\varphi_{n+1}^U(u')\} $. Using this and some simple calculation yields that 
\begin{equation}\label{def:Fgamma_nendpoints}
\begin{split}
\Gamma_n(\oa{uv})=  \max\Biggl\{ \left ( \frac{\varphi_n^U(v)-\varphi_{n}^L(v)}{\varphi_n^U(u)-\varphi_n^L(u)} \right ) \left (\varphi_{n+1}^L(u')-\varphi_{n}^L(u) \right )- \left(\varphi_{n+1}^L(v')-\varphi_{n}^L(v) \right ),\\  \left ( \frac{\varphi_n^U(v)-\varphi_n^L(v)}{\varphi_n^U(u)-\varphi_n^L(u)} \right ) \left (\varphi_{n+1}^U(v')-\varphi_n^L(u) \right )- \left (\varphi_{n+1}^U(v')-\varphi_n^L(v) \right ): \\
\oa{u'v'}\in G_{n+1}, \Psi_n(u')=u,  \Psi_n(v')=v, t \in [\varphi_{n+1}^L(u'),\varphi_{n+1}^U(u')]\Biggl\}. 
\end{split}
\end{equation}
\begin{equation}
\Gamma_n:=\max\{\Gamma_n(\oa{uv}):  \oa{uv}\in G_n\}.
\end{equation}

We will say that $\cF$-system $\fsystem$ satisfies \emph{Condition $\Gamma$} if $\sum^{\infty}_{n=0} \Gamma_n < \infty$.

 For $n \in \N$ we define  \[F_n =  \left \{(x,t): x \in X,  \ t \in [\varphi^L_n(x(n)), \varphi^U_n(x(n))]  \right \} = \bigcup_{g \in G_n} [g] \times [\varphi^L_n(g), \varphi^U_n(g)]. \] 
 Note that 
 \[ \bigcap_{n\in \N} F_n = {\bf F}_{\Phi} =  \left \{(x,t): x \in X,  \ t \in [\varphi^L(x), \varphi^U(x)]  \right \}.  \]  Now define $\tilde{s}_n:F_n  \rightarrow \mathbb{R}$ by 
\[\tilde{s}_n(x, t ):= s_n(u, v)(t) \ \textit{ for } t \in [\varphi^L_n(u), \varphi^U_n(u)]\]
where $u:=x(n)$ and $v:=H_X(x)(n)$. Finally, we define $$s:{\bf F}_{\Phi}  \rightarrow \mathbb{R}$$
\begin{equation}\label{eq:s(x)_new}
s(x, t ):= \lim_{n\to \infty} \tilde{s}_n (x, t).
\end{equation}
By Condition $\Gamma$, we have that $\tilde{s}_n$ converge uniformly to $s$ on ${\bf F}_{\Phi}$.

The above will aid us in constructing a continuous surjection from ${\bf F}_{\Phi}$ to ${\bf F}_{\Phi}$. In order to make it a homeomorphism, we need a condition stronger than $\Gamma$, namely Condition $\Gamma^+$
 defined as follows:
 \begin{equation}\label{def:Fgamma_n+}
\begin{split}
\Gamma^+_n(\oa{uv}):=  \max\Biggl\{ \left | s_n (v, u) (t)   - s_{n+1} (v',u') (t) \right | : \\
\oa{u'v'}\in G_{n+1}, 
\Psi_n(u')=u,  \Psi_n(v')=v
,
t \in [\varphi_{n+1}^L(u'),\varphi_{n+1}^U(u')]\Biggl\}. 
\end{split}
\end{equation}
\begin{equation}
\Gamma^+_n:=\max\{\Gamma^+_n(\oa{uv}):  \oa{uv}\in G_n\}.
\end{equation}
We will say that $\cF$-system $\fsystem$ satisfies \emph{Condition $\Gamma^+$} if $\sum^{\infty}_{n=0} \Gamma^+_n < \infty$.

Define $\tilde s^{+}_n:F_n  \rightarrow \mathbb{R}$ by 
\[\tilde{s}^{+}_n(x, t ):= s_n(v, u)(t) \ \textit{ for } t \in [\varphi^L_n(v), \varphi^U_n(v)]\]
where $u:=x(n)$ and $v:=H_X(x)(n)$. Finally, we define $$s^{+}:{\bf F}_{\Phi}  \rightarrow \mathbb{R}$$
\begin{equation}\label{eq:s+(x)}
s^{+}(x, t ):= \lim_{n\to \infty} \tilde{s}^{+}_n (x, t).
\end{equation}
By Condition $\Gamma$, we have that $\tilde{s}^{+}_n$ converge uniformly to $s^{+}$ on ${\bf F}_{\Phi}$.

Note that for fixed $n$ and $u,v \in G_n$, $s_n(u,v)$ and $s_n(v,u)$ are inverses of each other. Hence $\tilde s_n(x, \cdot)$ and $\tilde s_n^+(x, \cdot)$ are inverses of each other for fixed $n$ and $x \in \cantor$, implying that $s(x, \cdot)$ and $s^+(x, \cdot) $ are inverses of each other for fixed $x \in \cantor$. In particular, when both $\Gamma$ and $\Gamma ^+$ are satisfied, we have that $s(x,\cdot)$ is one-to-one.

\begin{theorem}\label{thm:main}
Let $\fsystem$ be an $\cF$-system satisfying condition $\Gamma$ and  ${\bf F}_{\Phi}$ be the associated fence. Then, there exists a continuous surjection $T: {\bf F}_{\Phi}\rightarrow {\bf F}_{\Phi}$, with $\homeocantor$ as a factor, satisfying \[T\left(x, t \right)=(\homeocantor(x), s(x,t)).\]
 Moreover, 
\begin{enumerate}  
    \item if $\fsystem$ additionally satisfies Condition $\Gamma^+$ and $\homeocantor$ is a homeomorphism of $X$, then $T$ is a homeomorphism of ${\bf F}_{\Phi}$.
\end{enumerate}
\end{theorem}
\begin{lemma}\label{lem:odometerminimal}
Suppose that a digraph $G$ is a cycle, 
   $\varphi^L, \varphi ^U: G \rightarrow [0,1]$ with $\varphi^L< \varphi^U$ and $\varepsilon >0$. Then, there exists a digraph $\tilde G$ which is a cycle, $\Psi: \tilde G \rightarrow G$ a surjective edge preserving homomorphism, and $\tilde\varphi^L, \tilde\varphi^U:\tilde G \rightarrow [0,1]$ such that 
   \begin{enumerate}[label=(\alph*)]
       \item if $\tilde g \in \tilde G$ and $\Psi(\tilde g) =g$, then $\varphi^L(g) \le \tilde\varphi^L(\tilde g) < \tilde\varphi^U(\tilde g)\le \varphi^U (g)$,
       \item for all $g \in G$, there exists  $\tilde g \in \tilde G$ with $\Psi(\tilde g) = g$ such that $\varphi^L(g) = \tilde\varphi^L(\tilde g)$ and $ \tilde\varphi^U(\tilde g)= \varphi^U (g)$,
       \item for all $g\in G$ and all $a, b$ with $ \varphi^L(g) \le a < b \le   \varphi^U (g)$, there exists  $\tilde g \in \tilde G$ with $\Psi(\tilde g) = g$ such that $|a- \tilde\varphi^L(\tilde g)| < \varepsilon$ and $ |b- \tilde\varphi^U(\tilde g)| < \varepsilon$.
      \item Letting $G_n = G$ and $G_{n+1} = \tilde G$, we have that $\Gamma_n, \Gamma^+_n < \varepsilon$, as defined in Equation~\eqref{def:Fgamma_n_new} and \eqref{def:Fgamma_n+}. In other words, whenever $\oa{\tilde g_1 \tilde g_2} \in \tilde G$ with $\Psi(\tilde g_i) = g_i$, the affine map $\alpha_1$ which takes $[\varphi^L(g_1), \varphi^U(g_1)]$ to $[\varphi^L(g_2), \varphi^U(g_2)]$ is $\varepsilon$-close to the affine map $\alpha_2$ which takes $[\tilde\varphi^L(\tilde g_1), \tilde\varphi^U(\tilde g_1)]$ to $[\tilde\varphi^L(\tilde g_2), \tilde\varphi^U(\tilde g_2)]$ on their common domain and also $\alpha_1^{-1}$ is $\varepsilon$-close to $\alpha_2^{-1}$ on their common domain.  
      \end{enumerate}
      Moreover, there exists $m_0 \in \N$ such that for all $m \ge m_0$ we can choose $\tilde G$ so that  $|\tilde G|=m|G|$.
\end{lemma}

\begin{proof}
    For each $g \in G$ let $I_g =[\varphi^L(g),\varphi^U(g)]$ and $\{x_g(j)\}_{j=0}^N$ be a sequence that partitions $I_g$ into equal pieces. We will describe how to choose $N$ later so that we can focus on the general construction now.  We want an ordering of $\{(x_g(j),x_g(k))\}_{g \in G,\  0 \le j < k \le N} $. More precisely, we can construct a sequence $\{(h_i, j_i,k_i))\}_{i=1}^L$ so that
    \begin{itemize}
        \item [(i)] $h_i \in G$, $0 \le j_i < k_i 
        \le N$,
        \item[(ii)] for all $1 \le i <L$, $\oa{h_ih_{i+1}} \in G$,
        \item[(iii)] for all $g \in G$, for all $0 \le j <k \le N$, there exists $1 \le i \le L$ such that $(h_{i},j_{i},k_{i}) = (g, j, k)$,
         \item[(iv)] $\oa{h_Lh_0} \in G$, $j_0=j_L=0$ and $k_0= k_L=N$.
        
    \end{itemize}
    The idea is that we start with some $g \in G$, $1 \le j_1 < k_1 \le  N$ and let $h_1 =g$, and go around the digraph $G$, changing $g$ according to the edge structure and covering all possible $0 \le j <k \leq N$.

    Note that necessarily $L$ is a multiple of $|G|$. 
    We let $\tilde{G}$ be any set consisting of $N$ distinct objects $\{\tilde{g}_1, \tilde{g}_2, \ldots, \tilde{g}_N\}$ different from elements of $G$. Digraph structure on set $\tilde{G}$ is defined in a natural way: the only edges of $\tilde{G}$ are $\oa{\tilde{g}_i\tilde{g}_{i+1}}$,  $1 \le i <N$ and $\oa{\tilde{g}_N \tilde{g}_1}$, yielding that $\tilde{G}$ is a cycle. The map $\Psi: \tilde{G} \rightarrow G$ is defined in the obvious way: $\Psi(\tilde{g}_i) = h_i$. Conditions (ii) and (iv) above, guarantee us that $\Psi$ is an surjective edge preserving homomorphism onto $G$. The maps $\tilde{\varphi}^L, \tilde{\varphi}^U: \tilde{G} \rightarrow [0,1]$ are defined by $\tilde{\varphi}^L (\tilde{g}_i)=x_{h_i}(j_i)$, $\tilde{\varphi}^U (\tilde{g}_i)=x_{h_i}(k_i)$. It is clear that required Conditions (a) and (b) of the lemma  are satisfied. If we choose $N$ large enough so that $1/N < \varepsilon$, then Condition (c) is satisfied as well.

    We also need to satisfy part (d) of the lemma. The above construction does not do this. We modify $\tilde {G}$ by inserting cycles between $\tilde{g_i}$ and $\tilde{g}_{i+1}$, for all $1 \le i < L$. More precisely, fix $1 \le i < L$. We let $\{c^i_j\}_{j=1}^{M_i}$ be a sequence such that 
    \begin{itemize}
        \item $c^i_1 = \tilde{g}_i$, $c^i_{M_i} = \tilde{g}_{i+1}$,
        \item $\Psi$ is defined on $\{c^i_j\}_{j=1}^{M_i}$ so it is a surjective edge preserving homomorphism onto $G$.
        \item $\tilde{\varphi}^L$ and $\tilde{\varphi}^U$ are  extended on $\{c^i_j\}_{j=1}^{M_i}$ so that for all $1 \le j < M_i$, we have that 
        \begin{itemize}
            \item the affine map $\alpha^i_j$ which takes $[\tilde{\varphi}^L(c^i_j), \tilde{\varphi}^U(c^i_j)]$ to $[\tilde{\varphi}^L(c^i_{j+1}), \tilde{\varphi}^U(c^i_{j+1})]$ is $\varepsilon$-close to the affine map $\beta^i_j$ which takes $[\varphi^L(\Psi(c^i_j)), \varphi^U(\Psi(c^i_j))]$ to\\ $[\varphi^L(\Psi(c^i_{j+1})), \varphi^U(\Psi(c^i_{j+1}))]$ on their common domain 
            \item  the inverse of $\alpha^i_j$ is $\varepsilon$-close to the inverse of $\beta^i_j$ on their common domain.  
        \end{itemize}
    \end{itemize} 
    The last condition can be satisfied by small perturbations of  $[\varphi^L(\Psi(c^i_j)), \varphi^U(\Psi(c^i_j))]$ to\\ $[\varphi^L(\Psi(c^i_{j+1})), \varphi^U(\Psi(c^i_{j+1}))]$ and  using sufficiently large $M_i$.
    
We replace the previous $\tilde{G}$ by the above modified $\tilde{G}$ and label it $\{\tilde{g}_1, \ldots, \tilde{g}_M\}$. Note that $M$ is necessary a multiple of $L$ and hence a multiple of $|G|$. For any $m_0 \ge M$ we can make $|\tilde{G}| = m |G|$ by simply extending $\tilde{G}$ by going around $G$ once more with $\tilde{\varphi}^U$ and 
$\tilde{\varphi}^L$ constant as the last step. Note that Condition (iv) above allows us to do this. 
    
\end{proof}

\begin{theorem}\label{thm:minimalfraisse} Given an odometer system $(X, \homeocantor)$, there exists a minimal homeomorphism of the Fra\"iss\'e fence which admits $\homeocantor$ as a factor.  
\end{theorem}
\begin{proof}
Fix a $p$-adic odometer where $p=(p_i)$, $p_i| p_{i+1}$ and $p_i<p_{i+1}$, $i\in\N$. Note that every subsequence of $(p_i)$ induces an odometer conjugate to the original one.
    We repeatedly apply Lemma~\ref{lem:odometerminimal} to construct an $\cF$-system $\fsystem$.  Suppose step $n$ is constructed and we have defined $(G_n, \Psi_n, \varphi ^L_n, \varphi^U_n)$. At step $n+1$ we choose $m=p_i$ for sufficiently large $i$, indicated by Lemma~\ref{lem:odometerminimal} applied to $(G_n, \Psi_n, \varphi ^L_n, \varphi^U_n)$, and $\epsilon <2^{-n}$. Applying Lemma~\ref{lem:odometerminimal}, we obtain  $(G_{n+1}, \Psi_{n+1}, \varphi ^L_{n+1}, \varphi^U_{n+1})$. Conditions (a) and (b) ensure that $\fsystem$ is a $\cF$-system. Condition (d) implies that $\fsystem$ satisfies Conditions $\Gamma$ and $\Gamma^+$. 
    Note that the induced homeomorphism $\homeocantor$ is conjugated to the $p$-adic odometer.
    By Theorem \ref{thm:main}, there is a homeomorphism  $T$ on the resulting fence $\fence$ which lifts $\homeocantor$. Condition (c) of Lemma~\ref{lem:odometerminimal} implies that $\eta_n$ as defined in \eqref{eq:eta_n} goes to zero. Consequently, Theorem~\ref{thm:fenceconst} (5) guarantees us that $\fence$ is a Fra\"iss\'e fence.  By Theorem~\ref{thm:GenFraisse} (2) map $T$ is minimal.
\end{proof}

\begin{corollary}\label{cor:Fraisse minimal}
There exist uncountably many minimal homeomorphisms on the Fra\"iss\'e fence such that none is a factor of any other.
\end{corollary}

\begin{proof}
By \cite[Proposition 4.5 (1)]{Kurka}, a $p$-adic odometer is a factor of a $q$-adic odometer if and only if for every $i\in\N$ there exists $j\in\N$ such that $p_i$ divides $q_j$, where $p=(p_k)$ and $q=(q_k)$. Since there is an uncountable family of infinite subsets of prime numbers, each pair of which intersects in a finite set, we can easily construct uncountably many odometers $H_a, a\in A$, such that none is a factor of any other. 

By Theorem \ref{thm:minimalfraisse} we can find a lift $T_a$ on the Fra\"iss\'e fence of $H_a$, $a\in A$. We claim that $T_a$ is not a factor of $T_b$ for $a\neq b\in A$.
Suppose that $T_a$ is a factor of $T_b$. Then $H_a$ is a zero-dimensional factor of $T_b$. Since $H_b$ is the maximal zero-dimensional factor of $T_b$, it follows that $H_a$ is a factor of $H_b$, yielding a contradiction. In the above, we used the fact that every dynamical system $(X,f)$ has a unique maximal factor, namely the factor determined by the decomposition of the space into its components. In our case, the maximal factor of $T_b$ is $H_b$. 
\end{proof}

\section{Acknowledgments}
 J.~\v Cin\v c was partially supported by Slovenian research agency ARIS grant J1-4632 and ARIS project under Contract No. SN-ZRD/22-27/0552. 
 J.~\v Cin\v c and U.~Darji acknowledge the support of the ARIS grant J1-4632, which enabled U.~Darji to visit the University of Maribor, where this project was initiated, and supported his subsequent visits.

 B.~Vejnar was supported by the grant GA\v{C}R 24-10705S.

\begin{table}[ht]
\begin{tabular}[t]{p{2.5cm}  p{11cm} }
\includegraphics [width=2.1cm]{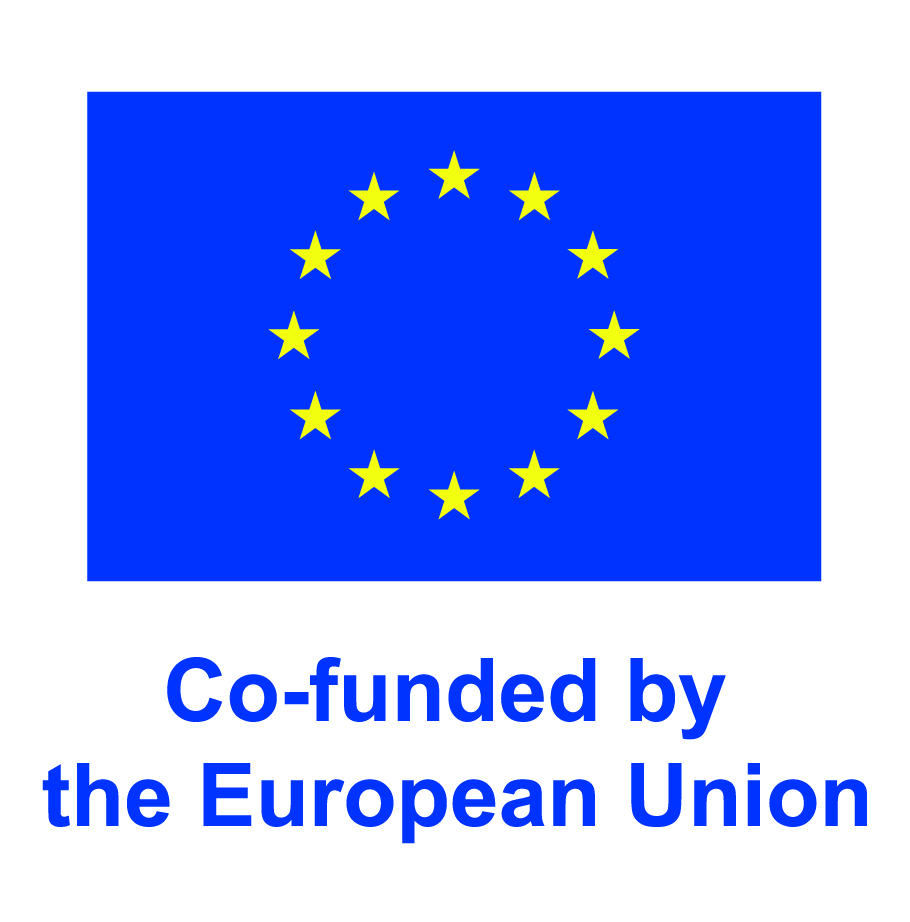} &
\vspace{-2cm}
This research is part of J.\ \v Cin\v c's  project that has received funding from the European Union's Horizon Europe research and innovation programme under the Marie Sk\l odowska-Curie grant agreement No.\ HE-MSCA-PF-PFSAIL-101063512.\\
\end{tabular}
\end{table}

\bibliographystyle{alpha}
\bibliography{citations}
\end{document}